\documentclass[12pt]{amsart}

\usepackage{amsmath}
\usepackage{amssymb,amsfonts,mathrsfs, amsthm, stmaryrd, color}
\usepackage[margin=2 cm,heightrounded=true,centering]{geometry}

\usepackage{amscd}
\usepackage{verbatim}
\usepackage{euscript}

\newtheorem{theorem}{Theorem}[section]
\newtheorem{proposition}[theorem]{Proposition}
\newtheorem{lemma}[theorem]{Lemma}
\newtheorem{remark}[theorem]{Remark}
\newtheorem{corollary}[theorem]{Corollary}

\newcommand{\cal}{\mathcal}

\newcommand{\gbar}{\overline{g}}
\newcommand{\gtil}{\widetilde{g}}

\newcommand{\bS}{\mathbb{S}}

\newcommand{\bN}{\mathbb{N}}
\newcommand{\sE}{\mathcal{E}}

\DeclareMathOperator{\inj}{inj}
\DeclareMathOperator{\conj}{conj}

\begin{document}

\title[Normalized Ricci flow and rotational symmetry]{Asymptotically hyperbolic normalized Ricci flow and rotational symmetry}

\author{Eric Bahuaud}
\address{Department of Mathematics,
Seattle University,
901 12th Ave,
Seattle, WA 98122, United States}

\email{bahuaude(at)seattleu.edu}

\author{Eric Woolgar}
\address{Dept. of Mathematical and Statistical Sciences, and Theoretical Physics Institute, University of Alberta, Edmonton, AB, Canada T6G 2G1.}

\email{ewoolgar(at)ualberta.ca}

\date{\today}

\begin{abstract}
\noindent We consider the normalized Ricci flow evolving from an initial metric which is conformally compactifiable and asymptotically hyperbolic. We show that there is a unique evolving metric which remains in this class, and that the flow exists up to the time where the norm of the Riemann tensor diverges. Restricting to initial metrics which belong to this class and are rotationally symmetric, we prove that if the sectional curvature in planes tangent to the orbits of symmetry is initially nonpositive, the flow starting from such an initial metric exists for all time. Moreover, if the sectional curvature in planes tangent to these orbits is initially negative, the flow converges at an exponential rate to standard hyperbolic space. This restriction on sectional curvature automatically rules out initial data admitting a minimal hypersphere.
\end{abstract}

\maketitle
%\tableofcontents

\section{Introduction}
\setcounter{equation}{0}

\noindent Conformal compactification techniques have proved to be very helpful in general relativity, where they were pioneered by Roger Penrose \cite{Penrose}. Subsequently they were also found to facilitate the study of Riemannian manifolds with asymptotic ends, especially asymptotically hyperbolic manifolds, on which the sectional curvatures approach $-1$ at infinity. The most important case is that of complete \emph{Poincar\'e-Einstein} manifolds. A Poincar\'e-Einstein manifold is one whose metric is conformally compactifiable and Einstein. These manifolds were studied in a seminal paper of Fefferman and Graham \cite{FG} and play an important role in the AdS/CFT correspondence of string-theoretic physics. This line of enquiry has led to important advances in conformal geometry, yielding hitherto unknown invariants for the conformal metric carried on the boundary-at-infinity associated to the conformal compactification.

There are many unanswered and tantalizing questions about Poincar\'e-Einstein (PE) manifolds and the related but more general class of \emph{asymptotically Poincar\'e-Einstein} (APE) manifolds \cite{BMW}. For example, can one construct a non-constant curvature PE manifold whose conformal infinity is a connected, compact hyperbolic space? Of particular interest for physics, can one construct an APE metric with this conformal infinity which obeys the so-called \emph{static Einstein equations} with non-constant lapse function?  The existence of such metrics is not even confirmed, and an analytic construction would likely be difficult.

It is thus conceivable that the detailed study of these metrics will proceed mostly by numerical techniques.\footnote
{A similar situation arises for Calabi-Yau 3-folds; see \cite{Douglas} for a discussion.}
Figueras and Wiseman \cite{FW} have pursued the numerical construction of Poincar\'e-Einstein metrics using a relaxation method based on the Ricci flow (see \cite{Wiseman} for an overview), suitably normalized to account for sectional curvatures which asymptote to $-1$. These numerical results have already had significant impact in the physics of the Randall-Sundrum model. Yet the rigorous analytical analysis of the underlying flow technique is not well-developed. While the relevant physics may be too complicated to study effectively without numerics, the successful use of the flow technique in this numerical approach and the important physics conclusions which follow from it provide compelling reasons for mathematicians to study the long-time existence and convergence behaviour of the flow in those settings which are simple enough to address analytically.

The flow in question is the \emph{normalized Ricci flow} (NRF) on an $n$-manifold $({\cal M},g)$, given by the evolution equation
\begin{equation}
\label{eq1.1}
\frac{\partial g}{\partial t} = -2 \left ( {\rm Ric}(g) + (n-1)g \right ) =: -2E(g)\ .
\end{equation}
The analysis of this flow was initiated in \cite{Bahuaud}, wherein the short-time existence in the asymptotically hyperbolic category was proved. Qing, Shi, and Wu \cite{QSW} independently studied this flow, and in addition to the short-time existence also obtained long-time existence if the initial curvature was sufficiently pinched.

This raises the question of long-time existence for ``large data''; i.e., when the curvature pinching is removed. Obviously this is a very difficult question in general, but ought to be tractable under conditions of high symmetry at least. Apart from its mathematical interest, there is already much to be explained in the numerical work discussed above, which rigorous convergence results could address. In the present paper we study long-time existence (and convergence) of this flow for the simplest nontrivial case of a complete metric evolving according to equation (\ref{eq1.1}), namely that which starts from initial data that are rotationally symmetric and asymptotically hyperbolic.

Our approach is to study the long-time existence of the flow (\ref{eq1.1}) for rotationally symmetric metrics that asymptotically hyperbolic in a weak sense.  These metrics have curvatures approaching $-1$ near infinity but require no additional smoothness of a conformal compactification.

We begin by considering the general short-time existence of solutions of the normalized Ricci flow within the asymptotically hyperbolic category, without any symmetry assumptions. It has been known since the work of WX Shi \cite{Shi} that for given smooth initial data on a complete manifold there is a unique Ricci flow evolution of this data for an interval of time $0\le t <T$ (for some $T\in (0,\infty ]$) such that the curvature is bounded (by a $T$-dependent bound), a result that also applies to normalized Ricci flow. In Section 2 we collect a number of facts that imply, if the initial data are conformally compact and asymptotically hyperbolic, that there is a unique asymptotically hyperbolic evolution $g(t)$ under normalized Ricci flow on this time interval. We show that the asymptotically hyperbolic evolution can be continued to $t=T$ and beyond unless the norm of the Riemann curvature diverges as $t\nearrow T$. Let $K$ be the type-$(0,4)$ algebraic curvature tensor
\begin{equation}
\label{eq1.2}
K_{ijkl} = g_{il} g_{jk} - g_{ik}g_{jl}.
\end{equation}
We prove

\begin{theorem}[AH existence and unique continuation]
\label{theorem1.1}
Suppose that $g_0$ is smoothly conformally compact and asymptotically hyperbolic. Then the normalized Ricci flow has a unique solution which remains conformally compact and asymptotically hyperbolic in $[0,T_M)$. If $T_M <  \infty$, then
\begin{equation}
\label{eq1.3}
\limsup_{t \nearrow T_M} \sup_{p \in {\cal M}} | {\rm Rm} +K |_g = \infty.
\end{equation}
\end{theorem}

The important point here is that conformally compact asymptotically hyperbolic metrics have a certain asymptotic expansion near infinity. One must show first that initial data which have such an expansion will always yield a unique flow with bounded curvature, at least for a finite interval of time, and next that the form of the expansion is preserved for as long as the flow exists. The existence and uniqueness of a bounded curvature flow follows from earlier work \cite{Shi, CZ, Bahuaud}, and indeed the ingredients for the unique continuation are already in the literature.  We thus take this opportunity to explicitly state the unique continuation result.

Beginning in Section 3, we specialize to rotational symmetry. We consider a related, gauge-fixed flow for rotationally symmetric metrics in which the metric is expressed at all times in area-radius coordinates. This flow reduces to a single parabolic equation for a function $f$. This function describes the norm of the radial coordinate vector field for an evolving, rotationally symmetric metric. This metric can be pulled back to a normalized Ricci flow, which for the given initial data is the unique Ricci flow of Shi. Certain combinations of $f$ and its first radial derivative can be shown to be uniformly bounded along the gauge-fixed flow and these combinations pull back to the sectional curvatures of the evolving Ricci flow metric. Let $\lambda$ denote the sectional curvature in $2$-planes tangent to the orbits of rotational symmetry. We prove the following result.

\begin{theorem}[Long-time existence] \label{theorem1.2}
Let $g(t)$ be the unique normalized Ricci flow of bounded curvature on $[0,T_{\max})\times {\mathbb R}^n$ developing from an initial metric $g(0)=g_0$ such that $g_0$ is complete asymptotically hyperbolic, and rotationally symmetric. Let $T_{\max}$ be the maximal time of existence. Then
\begin{enumerate}
\item[(a)] $g(t)$ is complete, rotationally symmetric and asymptotically hyperbolic for all $t\in [0,T_{\max})$ and
\item[(b)] if the initial sectional curvature in 2-planes tangent to the orbits of rotational symmetry is $\lambda(g_0)\le 0$, then
\begin{enumerate}
\item[(i)] $T_{\max}=\infty$,
\item[(ii)] $\lambda(g(t))\le 0$ for all $t>0$, and
\item[(iii)] if $\lambda(g_0)< 0$, there are constants $C,c>0$ depending only on $n$ and $g_0$ such that $\vert {\rm Rm} + K \vert < Ce^{-ct}$.
\end{enumerate}
\end{enumerate}
\end{theorem}

\begin{remark}\label{remark1.3}
Assumption (b) above implies that the initial data contain no closed minimal hypersurfaces.
\end{remark}

\begin{proof}[Proof of Remark] Consider first a minimal hypersphere that is an orbit of the rotational symmetry group. But in rotational symmetry the Gauss-Codazzi equations for the orbits give that
\begin{equation}
\label{eq1.4}
\frac{1}{r^2}=\lambda+\frac{H^2}{(n-1)^2}\ ,
\end{equation}
where $r$ is the intrinsic radius of curvature of an orbit and $H$ is its mean curvature. When $\lambda\le 0$, then
\begin{equation}
\label{eq1.5}
\frac{1}{r^2}\le \frac{H^2}{(n-1)^2}\ ,
\end{equation}
so $H>0$, and then no orbit can be a minimal hypersphere.

The general result now follows from the standard barrier argument. A closed minimal hypersurface $\Sigma$ has a point $p$ which maximizes the distance from the origin of rotational symmetry. At $p$, $\Sigma$ is tangent to a distance sphere from the origin (and thus an orbit of rotational symmetry) bounding a closed ball $B$. By the above paragraph, this distance sphere is mean convex (indeed, by symmetry, convex). But by the maximum principle for minimal surfaces, any minimal hypersurface tangent to a mean convex surface at $p$ cannot lie entirely within $B$, contradicting that $p$ maximizes distance from the origin.
\end{proof}

This leaves open the possibility that for initial data containing no minimal hypersphere but having $\lambda$ somewhere positive, a neckpinch could form during the flow. This intriguing possibility cannot happen for rotationally symmetric, asymptotically flat Ricci flow \cite{OW}, and would be worthy of further study (possibly numerically).

By Theorem \ref{theorem1.2}, provided that the sectional curvatures in planes tangent to the orbits of symmetry are initially $\lambda< 0$, the norm of the Riemann tensor is not only uniformly bounded but all sectional curvatures decay to $-1$ exponentially along the flow. We are then able to show that the flow converges in the infinite-time limit to hyperbolic space.

\begin{theorem}[Convergence]\label{theorem1.4}
Let $g(t)$, $g_0$ be as in Theorem \ref{theorem1.2}. Choose any increasing, unbounded sequence $t_n$. Let $O\in {\cal M}$ denote the origin of rotational symmetry. Then $({\cal M},g(t_n),O)$ converges in the pointed Cheeger-Gromov sense to standard hyperbolic space.
\end{theorem}

This paper is organized as follows. Section 2 contains background on conformally compactifiable and asymptotically hyperbolic manifolds. We also present certain standard Ricci flow results, which we adapt to the normalized Ricci flow \ref{eq1.1} for use in the proof of Theorem \ref{theorem1.1} given at the end of that section. In Section 3, we specialize to rotational symmetry. We outline our strategy, which involves introducing a gauge-fixed, rotationally symmetric flow which pulls back to normalized Ricci flow but which has the virtue of reducing to a single parabolic flow equation, a strategy also used in earlier work on Ricci flow in asymptotically flat manifolds \cite{OW, GOW}. We discuss related work of Ma and Xu \cite{MX} in this section. We also list certain basic results that we will need later. Section 4 contains a parabolic maximum principle adapted to the singular PDEs which arise, based on the strong maximum principle of Hopf. We apply this principle to prove that no neckpinches can form under our rotationally symmetric flow if no minimal hypersphere is present in the initial data. In Section 5, we find bounds on the sectional curvature $\lambda$ in $2$-planes tangent to the orbits of the rotational symmetry. In Section 6, we find that the difference between $\lambda$ and the sectional curvature $\kappa$ in planes containing the radial direction decays exponentially to zero, provided $\lambda<0$. Combining these results with the continuation principle of Theorem \ref{theorem1.1}, the proofs of Theorems \ref{theorem1.2} and \ref{theorem1.4} follow easily. These are given in Section 7.

\subsection{Conventions} Our conventions generally follow those of the book \cite{Chow}. However, we lower the index of the Riemann tensor ${\rm Riem}=R_{kli}{}^j$ into the fourth position, so $g(\cdot,{\rm Riem})=g_{jp}R_{kli}{}^p=R_{klij}$. We sometimes denote this object (i.e., the Riemann tensor with four lower indices) by $\rm{Rm}$.

\subsection{Acknowledgements} EB is grateful to Fr\'ed\'eric Rochon for a discussion that clarified the proof of Theorem \ref{theorem1.1}. The work of EW was partially supported by NSERC Discovery Grant RGPIN 203614.  Both authors are grateful to the Park City Math Institute 2013 Summer Program, where this work was commenced, and the Banff International Research Station workshop 15w5148, \emph{Geometric Flows: Recent Developments and Applications}, where this work was completed.

\section{Flow of conformally compact asymptotically hyperbolic metrics}
\setcounter{equation}{0}

\subsection{Definitions.} A manifold $({\cal M},g)$ with an asymptotic end is said to be \emph{(smoothly) conformally compactifiable} if there is a (smooth) manifold-with-boundary $(M,{\tilde g})$, an embedding $i:{\cal M}\hookrightarrow M$ that maps ${\cal M}$ onto the interior of $M$, and a function $\Omega: M\to [0,\infty)$ such that $g=i^*{\left(\Omega^{-2} \tilde g\right)}$ where $\Omega$ vanishes to first order precisely on $\partial M$. Then $({\cal M},g)$ is \emph{conformally compact asymptotically hyperbolic} if there is a smooth boundary defining function $\Omega$, such that $|d \Omega|^2_{\Omega^2 g} = 1$ on $\partial M$ and $\Omega^2 g$ extends to a smooth metric on $\partial M$. For any representative metric $g_{(0)}$ of the induced conformal class on the boundary, we may solve for a new \emph{special defining function} $x$ so that $|dx|^2_{x^2 g} \equiv 1$ on a collar of the boundary. In this neighborhood, the metric may be decomposed as
\begin{equation}
\label{eq2.1}
g = \frac{dx^2 + g_x}{x^2},
\end{equation}
with $g_x$ a smooth family of metrics on $\partial M$ and $g_x|_{x=0} = g_{(0)}$.

Let $K$ denote the $+1$ constant curvature $4$-tensor
\begin{equation}
\label{eq2.2}
K_{ijkl} = g_{il} g_{jk} - g_{ik}g_{jl}.
\end{equation}
A computation shows that asymptotically hyperbolic metrics are complete and of bounded curvature; moreover
\begin{equation}
\label{eq2.3}
|{\rm Rm} + K|_g = O(x).
\end{equation}

\subsection{Normalized Ricci flow: Unique continuation}

\noindent We consider the NRF given by \eqref{eq1.1} in two related settings, that of complete metrics of bounded curvature, and that of conformally compact asymptotically hyperbolic metrics.

The following result allows us to apply certain standard Ricci flow results to the normalized Ricci flow.

\begin{lemma}\label{lemma2.1}
Define
\begin{equation}
\label{eq2.4}
\begin{split}
u(t):=&\, \frac{e^{2(n-1)t}-1}{2(n-1)}\ ,\\
{\hat g}\circ u :=&\, e^{2(n-1)t}g=\left ( 1+2(n-1)u\right )g\ .
\end{split}
\end{equation}
Then $g(t)$ is a normalized Ricci flow for $t\in [0,T]$ with initial metric $g(0)=g_0$ if and only if ${\hat g}(u)$ is a Ricci flow for $u\in [0,U]$ with initial metric ${\hat g}=g_0$, where $U=\frac{e^{2(n-1)T}-1}{2(n-1)}$.
\end{lemma}

\begin{proof}
Compute $\frac{\partial}{\partial t} {\hat g}\circ u= \frac{\partial}{\partial t}\left ( e^{2(n-1)t}g\right )$ and use (\ref{eq1.1}).
\end{proof}

We recall the fundamental existence result for the Ricci flow on non-compact manifolds:
\begin{theorem}[WX Shi \cite{Shi}]
\label{theorem2.2}
Let $(M,g_{ij}(x))$ be an $n$-dimensional complete noncompact Riemannian manifold with Riemann curvature ${\rm Rm}$ satisfying $|R_{ijkl}|^2 \leq k_0$ on ${M}$ for some $k_0 \in {\mathbb R}$. Then there exists a constant $T(n,k_0) > 0$ such that the evolution equation
\begin{equation}
\label{eq2.5}
\begin{split}
\partial_t g =&\, -2 {\rm Rc}\ ,\\
g(0)=&\, g_0\ ,
\end{split}
\end{equation}
has a smooth solution $g_{ij}(t,x) > 0$ for a short time $0 \leq t \leq T(n,k_0)$. Furthermore, for any integer $m \geq 0$, there exist constants $c_m = c_m(n, m, k_0)$ such that
\begin{equation}
\label{eq2.6}
\sup_{x \in {\cal M}} | \nabla^m R_{ijkl}(t,x)|^2 \leq \frac{C_m}{t^m}, \; 0 \leq t \leq T(n,k_0).
\end{equation}
\end{theorem}

Regarding uniqueness, we have the following:

\begin{theorem}[Chen and Zhu \cite{CZ}]
\label{theorem2.3}
Let $(M^n, g_{ij}(x))$ be a complete noncompact Riemannian manifold of dimension $n$ with bounded curvature. Let $g_{ij}(t,x)$ and $\gbar_{ij}(t,x)$ be two solutions to the Ricci flow on $[0,T] \times M$ with the same $g_{ij}(x)$ as initial data and with bounded curvatures. Then $g_{ij}(t,x) = \gbar_{ij}(t,x)$ for all $(t,x) \in [0,T] \times M^n$.
\end{theorem}

As a consequence, we immediately obtain the following

\begin{corollary}
\label{corollary2.4}
Theorems \ref{theorem2.2} and \ref{theorem2.3} apply with the flow (\ref{eq2.5}) replaced by (\ref{eq1.1}), and the estimates (\ref{eq2.6}) become
\begin{equation}
\label{eq2.7}
\sup_{x \in M} | \nabla_g^m R_{ijkl}(t,x)|^2 =  \sup_{x \in M} | e^{-2(n-1)t} \nabla_{ {\hat g}}^m  {\hat R}_{ijkl}(t,x)|^2
\leq \frac{C_m e^{-2(n-1)t}}{t^m}\ .
\end{equation}
\end{corollary}

\begin{proof}
Simply substitute $g = e^{-2(n-1)t} {\hat g}$ in (\ref{eq1.1}). The decay estimate (\ref{eq2.7}) follows from an elementary Riemann curvature scaling argument, which yields that ${\rm Rm}(g) = e^{-2(n-1)t} {\rm Rm}({\hat g})$. \end{proof}

As an AH metric $g_0$ is complete and of bounded curvature, the results above imply that there is a solution to the NRF starting at $g_0$ for a short time.  We call this \textit{Shi's flow}. Here is the basic long-time existence criterion for this flow [see \cite{Chow}, page 118].

\begin{theorem}\label{theorem2.5}
If $g_0$ is a smooth metric on a noncompact complete manifold with bounded curvature, the normalized Ricci flow has a unique solution $g(t)$ with $g(0) = g_0$ on a maximum time interval $0 \leq t < T_M \leq \infty$.  If $T_M < \infty$, then
\begin{equation}
\label{eq2.8}
\limsup_{t \to T_M} \sup_{p} |{\rm Rm} + K|_g = \infty.
\end{equation}
\end{theorem}

Theorem \ref{theorem1.1} concerns the NRF starting at a conformally compact asymptotically hyperbolic metric. The results above alone do not imply that the short time solution to the NRF remains within this class. However, by earlier work of the first author, the unique solution to the NRF remains conformally compact AH for a short time. One could worry that $g(t)$ ceases to be conformally AH at time earlier than the maximal time of existence as a Shi flow. Theorem \ref{theorem1.1} guarantees that this does not happen. We now prove that theorem.

\begin{proof}[Proof of Theorem \ref{theorem1.1}]
In the statement of the theorem, the time $T_M$ represents the maximal interval of existence of the NRF as a Shi flow.  Let $T_1>0$ denote the maximal interval of existence within the class of smoothly conformally compact asymptotically hyperbolic metrics.  Suppose by way of contradiction that $T_1 < T_M$.

On $[0,T_1]$, $g(t)$ is a complete metric with uniformly bounded curvature, and so by the work of Chen-Zhu \cite{CZ} (see in particular Theorem 2.1 and observe that the estimates allow us to continue a solution to $[0,T_1]$), there exists a solution to the harmonic map heat flow coupled with the Ricci flow on $[0,T_1]$ that produces a solution $\gtil(t)$ to the normalized Ricci DeTurck flow on $[0,T_1]$ with uniform control of curvature and all covariant derivatives.  But now since $\gtil(0) = g_0$ is conformally compact asymptotically hyperbolic the regularity results in \cite[Section 5]{Bahuaud} give that $\gtil(t)$ is smoothly conformally compact AH on $[0,T_1]$.  Pullback by the DeTurck diffeomorphism now extends the conformally compact AH solution of the NRF on $[0,T_1]$, contradicting the maximality of $T_1$.  So $T_1 = T_M$.

The conformally compact flow now coincides with the Shi flow, and so the previous extension theorem applies.
\end{proof}

Note that this theorem says nothing about the regularity of conformal compactification of a \emph{limiting} Poincar\'e-Einstein metric. Should the flow converge to a limit, nothing here requires the limit metric to be smoothly conformally compact.

%%%%%%%%%%%%%%%%%%%%%%%%%%%%%%%%%%%%%

\section{Rotational symmetry}
\setcounter{equation}{0}

\subsection{Introduction}

\noindent There are a number of approaches to rotationally symmetric Ricci flow, see for example \cite{AngenentKnopf, Ivey, MX, OW}. A time-dependent rotationally symmetric metric can be parametrized as
\begin{equation}
\label{eq3.1}
g=\phi^2(t,\rho) d\rho^2+\psi^2(t,\rho)g(\bS^{n-1},{\rm can})\ .
\end{equation}
The Ricci flow of this metric leads to a system of two equations for $\phi$ and $\psi$ which, as usual, is not strictly parabolic. One can use the DeTurck trick, but for long-time existence this is inconvenient since it leads to a parabolic system with two unknown functions. Instead, other methods are available which effectively reduce the problem to a single parabolic equation for a single function.

One such method is to fix a coordinate gauge, and modify the Ricci flow by a diffeomorphism term to preserve the gauge during the flow. The two obvious choices are area-radius coordinates (in which $\psi(t,\rho)=\rho$ for all $t\ge 0$) and normal coordinates about the origin of rotational symmetry (in which $\phi(t,\rho)=1$ for all $t\ge 0$). The former leads to a single parabolic equation for $\psi$, while the latter leads to a considerably more complicated coupled system which we will not consider further. However, a variant on the normal coordinate strategy leads to a common approach in the literature, in which the coordinate vector field $\frac{\partial}{\partial \rho}$ is replaced by $\frac{\partial}{\partial s}:=\frac{1}{\varphi}\frac{\partial}{\partial \rho}$. Then $\frac{\partial}{\partial s}$ and $\frac{\partial}{\partial t}$ form a non-commuting set of vector fields on $[0,T)\times { M}$ (the partial derivative notation notwithstanding). This approach decouples $\phi$ from the $\psi$ equation, leaving a single parabolic equation for $\psi$; the solution for $\phi$ can be found after the $\psi$ equation is solved.

While both the area-radius gauge approach and the non-commuting vector fields approach lead to a single parabolic differential equation, in each case the resulting equation has a singular point at $\rho=0$, as one expects when employing polar coordinates. However, this may be more than just a coordinate artifact since, if sectional curvatures were to become unbounded, symmetry would suggest that this will occur at the origin. This issue does not arise, or at least is not central, in treatments of short-time existence or in attempts to find exact, rotationally symmetric solutions, but for us it will consume the bulk of the effort in what follows.

In the sequel, we have chosen the area-radius approach. This approach at first glance seems to have a disadvantage. Area-radius coordinates break down at minimal hyperspheres, so this approach is valid only when there are none. We will, however, show that if the initial data contain no minimal hyperspheres then none form during the evolution. Hence, this is a restriction on the initial data, and as noted by Remark \ref{remark1.3} it is a consequence of a restriction that we already impose on the sectional curvature.

Were we to employ the alternative approach of non-commuting vector fields, it would be unlikely that we could remove this restriction and thereby allow minimal hyperspheres in the initial data,\footnote
{Perhaps we could \emph{relax} the restriction without removing it entirely. However, we are able to study the $\lambda(0)\le 0$ (as opposed to $\lambda(0)<0$) case with our current methods, and cannot find evidence for convergence for data with $\lambda(0)$ not strictly negative. This may be (weak) evidence that $\lambda<0$ may be necessary for convergence.}
even though the issue of coordinate breakdown at minimal hyperspheres would no longer arise. To see this, consider the following gedankenexperiment, posed in the $n=3$ case for simplicity. Consider ${\mathbb R} \times {\mathbb S}^2$ with an asymptotically hyperbolic (at both ends) rotationally symmetric metric; a time-symmetric slice of a Kottler (i.e., AdS-Schwarzschild) metric will do. This contains a topologically essential minimal sphere, a neck. Under normalized Ricci flow, the neck may at first expand, but a standard argument from the analysis of the Ricci flow of closed 3-manifolds with nontrivial second homotopy can be adapted to this setting and shows that the neck eventually will contract and will form a neckpinch in finite time (see, e.g., \cite[pp 420--429, especially Lemma 18.11]{MT}). Allow the contraction to proceed almost to the singular time, but stop the flow before singularity formation. The neck will now have area of order $\epsilon^2$ for a small positive $\epsilon$.

Now consider a $2$-sphere of radius much larger than $1/\epsilon^2$, lying in one of the two asymptotic regions. Perform a surgery that removes the asymptotic region beyond that sphere, and smoothly attach in its place a sufficiently large (and thus nearly flat) $3$-sphere with a disk removed. This will require some smoothing near the surgery region. The result is ${\mathbb R}^3$ with a very small minimal surface, and the geometry for quite some distance around that minimal surface is insensitive to the surgery.

Now restart the normalized Ricci flow. Pseudo-locality considerations, if they can be applied here, would suggest that the evolution near the neck will be nearly unaffected by the surgery, and since a singularity would have formed immediately had the flow not been stopped, the same will happen now.

It therefore seems plausible that the restrictions inherent in the area-radius approach may not be as onerous as they may at first seem---perhaps greater generality is possible, but plausibly some form of restriction will always be necessary if the flow is to exist for all $t>0$ and converge. The current form of these restrictions accommodates our approach in which the analysis becomes relatively straightforward (though not entirely so, since the singular point at the origin must be carefully dealt with).

We take this opportunity to mention related work by Ma and Xu \cite{MX}. These authors study the behavior of the un-normalized Ricci flow within a class of rotationally symmetric metrics of asymptotically hyperbolic type, using the non-commuting vector fields approach. Nonetheless, they require \emph{all} sectional curvatures to be negative, whereas we have no need to restrict the sectional curvature in radial 2-planes. It is thus implicit in their assumptions, as in ours, that there are no initial minimal hyperspheres. Their flow can be transformed to a flow of our type, but then the asymptotic sectional curvatures would vary in time (according to a simple scaling law). In this transformed picture, their work amounts to the study of the same flow equation and initial conditions, but with a different boundary (rather, asymptotic) condition. They find long-time existence and convergence. Their arguments appear to require a maximum principle for singular differential equations, but the precise nature of this principle is not made explicit in \cite{MX} (\emph{cf} our Propositions \ref{proposition4.1} and Corollary \ref{corollary4.2}, which perhaps can be used to close this apparent gap).

\subsection{Rotational symmetry in area-radius coordinates}
Our strategy here is similar to that of \cite[Section 4.1]{OW}, but the smoothly conformally compactifiable setting means that differences arise in the treatment of the asymptotics.

By Theorem \ref{theorem1.1} and using that Ricci flow (including normalized Ricci flow) preserves isometries, if we start from a rotationally symmetric metric $g_0=g(0)$, then we may write the solution to (\ref{eq1.1}) in the form of (\ref{eq3.1}) on some time interval $t\in[0,T_M)$ (possibly $T_M=\infty$). A computation shows that the normalized Ricci flow (\ref{eq1.1}) for the ansatz of (\ref{eq3.1}) devolves to
\begin{equation}
\label{eq3.2}
\begin{split}
\frac{\partial \phi}{\partial t} =&\,  \frac{(n-1)}{\psi\phi}\frac{\partial^2 \psi}{\partial \rho^2} -\frac{(n-1)}{\psi \phi^2} \frac{\partial \phi}{\partial \rho}\frac{\partial \psi}{\partial \rho}  -(n-1)\phi, \\
\frac{\partial \psi}{\partial t} =&\, \frac{1}{\phi^2}\frac{\partial^2 \psi}{\partial \rho^2}  -\frac{1}{\phi^3} \frac{\partial \psi}{\partial \rho} \frac{\partial \phi}{\partial \rho} + \frac{(n-2)}{\phi^2 \psi} \left ( \frac{\partial \psi}{\partial \rho}\right )^2- \frac{n-2}{\psi} -(n-1)\psi.
\end{split}
\end{equation}
As an initial condition, we may choose $\rho$ to be the distance from the centre of symmetry with respect to the $g_0$ metric, so $\phi(0,\rho) \equiv 1$. Since the flow is smooth, $\phi$ remains bounded and even as a function of $\rho$ near $\rho = 0$, and $\phi(t,0) = 1$ (all of these statements hold at least for a short time).  Similarly, $\psi(t,\rho) \sim \rho$ remains an odd function of $\rho$ near $\rho = 0$. Since $g_0$ is conformally compact AH, then $\psi(0,\rho) \sim e^{ \rho}$, as $\rho \to \infty$. By Theorem \ref{theorem1.1}, $g(t)$ remains conformally compact AH. This ensures that $\psi(t,\rho) \sim e^{ \rho}$ as $\rho \to \infty$ at least for a short time as well. From these bounds, it follows that distance with respect to $g(t)$ is comparable to distance with respect to $g_0$ so that $\Omega = e^{-\rho}$ is a boundary defining function for the conformal infinity.

In view of the discussion of the last subsection, we prefer the ansatz
\begin{equation}
\label{eq3.3} \hat{g}(t) = f^2(t,r)dr^2+r^2 g\left (\bS^{n-1},{\rm can}\right )
\end{equation}
for the evolving metric on $({\mathbb R}^n,g)$. We now derive the equation that $f$ must satisfy. By the change-of-variables formula, using the notation $\phi_t(r):=\phi(t,\rho)$, $\psi_t(r):=\psi(t,\rho)$, we have
\begin{equation}
\label{eq3.4}
f(t,r) = \frac{\phi_t\circ \psi_t^{-1}(r)}{ \left ( \partial_{\rho}\psi_t\right ) \circ \psi_t^{-1}(r)}\ .
\end{equation}
This change of variables arises from a diffeomorphism of ${\mathcal M}$. The next lemma discusses this, after which a subsequent lemma extends the diffeomorphism to $M$ such that the asymptotic structure is preserved.
\footnote
{Equation (\ref{eq3.4}) corresponds to \cite[equation (4.13)]{OW}, except that the latter has a minor error in its denominator: $\partial_r$ should be $\partial_{\rho}$.}

\begin{lemma}\label{lemma3.1}
Let $g(t)=\phi_t^2(\rho)d\rho^2+\psi_t^2(\rho)g(S^{n-1},{\rm can})$ be the smooth, rotationally symmetric, conformally compactifiable and asymptotically hyperbolic flow of metrics obeying (\ref{eq1.1}) on ${\mathbb R}^n$ developing from an initial metric $g_0=g(0)$ that has no minimal hypersphere. Let $F_t:{\mathcal M}\to {\mathcal M}$ be the map with components $(r,\theta^A)=F_t(\rho,\theta^A)$ such that ${\hat g}=\left ( F_t^{-1}\right )^*g(t)=f^2(t,r)dr^2+r^2 g(S^{n-1},{\rm can})$. Then $F_t$ is a smooth diffeomorphism and ${\hat g}$ is smooth.
\end{lemma}

\begin{proof}
Since we are working in polar coordinates, strictly speaking, the map $F_t$ defined above is defined on ${\mathcal M}$ punctured at the pole.  Since $F_t$ is the identity in the angular components, we must check that the radial map $\psi_t:(0,\infty)_{\rho}\to (0,\infty)_r$ is invertible and has odd parity in $r$ as $r \to 0$.  For convenience we extend $\psi_t$  by continuity to zero; i.e., we regard $\psi_t:[0,\infty)_{\rho}\to [0,\infty)_r$, however we only use that $\psi_t$ and its inverse are smooth on $(0,\infty)$, and we use their asymptotic behaviour.

Now it is elementary that the mean curvature of constant-$\rho$ hypersurfaces
\begin{equation}
\label{eq3.5}
H=\frac{(n-1)\partial_{\rho}\psi_t}{\phi_t\psi_t}=\frac{n-1}{rf}
\end{equation}
tends to infinity at a fixed point of rotational symmetry (i.e., the origin) and tends to $n-1$ on an asymptotically hyperbolic end. Furthermore, if $g_0$ has no minimal hypersphere, then either the mean curvature $H(0,\rho)=\frac{(n-1)\partial_{\rho}\psi_0}{\psi_0}$ of the constant-$\rho$ hypersurfaces is greater than $n-1$ or it achieves a minimum $H_{\rm min}>0$ at some finite $\rho$; in either case, it is bounded away from zero. Then so are both $H(t,\cdot)$ and $\frac{\partial_{\rho}\psi_t}{\psi_t}$ for some interval $t\in [0,T)$. Likewise, taking a smaller $T$ if necessary, we have $0<C\le \phi_t$ (using that $\phi_0=1$). Since $H(t,\rho):=\frac{(n-1)\partial_{\rho}\psi_t}{\phi_t\psi_t}>0$ we may conclude that $\frac{\partial_{\rho}\psi_t}{\psi_t}>0$ for $t\in [0,T)$. Then $\psi_t$ is monotonic in $\rho$. Hence $\psi_t$ is injective.

Surjectivity follows because $\psi_t(0)=0$ ($SO(n)$ fixes an origin in ${\mathbb R}^n$), $\psi_t(\rho)\to \infty$ as $\rho\to\infty$ (there are $SO(n)$ orbits of arbitrarily large area in a conformally compactifiable, rotationally symmetric manifold), and $\psi_t(\cdot )$ is continuous. Thus $\psi_t$ is invertible, and hence so is $F_t$,

Smoothness of $\psi_t^{-1}(r)$ on the open interval $(0,\infty)$ follows from the same simple inverse function theorem argument by differentiating the composition $\psi_t^{-1}( \psi_t(\rho)) = \rho$ that will be used to prove Lemma \ref{lemma3.3} below. Moreover, since $\psi_t$ is an odd function of $\rho$ as $\rho \to 0$, one may check that $\psi^{-1}_t(r)$ will be odd as $r\to 0$.

The explicit expression for $f(t,r)$ given in equation \eqref{eq3.4} and the parity conditions for $\phi_t$ and $\psi_t$ and its inverse now show that $f$ is an even function of $r$.  It is easy to check that the leading coefficient is $1$.

Finally, although the diffeomorphism $F_t$ is not defined at the pole, the resulting pullback metric will still be smooth.  Since $r$ is just the area-radius coordinate for $g(t)$, the condition that $g(t)$ is smooth at $r=0$ is precisely the condition that there is a coordinate system in which the metric at $r=0$ becomes the flat metric in polar coordinates ${\hat g}= dr^2 + r^2 g(S^{n-1},{\rm can})$ plus corrections given by even powers of $r$. Thus since $f$ has leading coefficient $1$ and is even in $r$, ${\hat g}(t)$ is smooth across $r=0$.
\end{proof}

Thus the flow (\ref{eq3.3}) arises from (\ref{eq3.1}) by pullback along the inverse of the time-dependent diffeomorphism $(r, \theta^A) = F_t(\rho, \theta^A) = (\psi_t(\rho), \theta^A)$. Let the generator of the inverse diffeomorphism flow be denoted by $X_t \equiv X(t,\cdot)$. This is not a DeTurck type vector field defined by the difference between the connection of the flowing metric and a fixed background connection. Since $F_t$ is constant in time in the angular variables, $X_t$ will have vanishing angular components and
\begin{equation}
\label{eq3.6}
X_t = - \left ( \frac{\partial \psi_t}{\partial t} \circ {\psi_t^{-1}(r)}\right ) \frac{\partial}{\partial r}\ .
\end{equation}
From this and equations (\ref{eq3.2}) it follows that
\begin{equation}
\label{eq3.7}
X_t =  \left[ \frac{1}{f^3(r)}\frac{\partial f}{\partial r} + \frac{(n-2)}{r}\left ( 1-\frac{1}{f^2(r)}\right ) + (n-1) r\right] \frac{\partial}{\partial r}.
\end{equation}
While individual terms in $X_t$ are singular at $r=0$, from the proof of Lemma \ref{lemma3.1}, we have that $f\sim 1+c(t)r^2+\dots$ as $r\to 0$, and so $X_t$ is not singular. Rather, it vanishes at the origin, as it must. Moreover, if the flow (\ref{eq3.1}) solves (\ref{eq1.1}), then the metric $\hat{g} = (F_t^{-1})^* g$ of (\ref{eq3.3}) solves
\begin{equation}
\label{eq3.8}
\frac{\partial \hat{g}}{\partial t} = -2 \left ( {\rm Rc} ({\hat g}) +(n-1)\hat{g}\right ) +\pounds_{X_t}{\hat g} \ .
\end{equation}
The angular components are constant in time, thus preserving ansatz (\ref{eq3.3}) (providing another derivation of (\ref{eq3.7})). The evolution of the $rr$-component of the metric is given by
\begin{equation}
\label{eq3.9}
\begin{split}
\frac{\partial f}{\partial t} =&\, \frac{1}{f^2} \frac{\partial^2 f}{\partial r^2}
-\frac{2}{f^3} \left ( \frac{\partial f}{\partial r} \right )^2 +\left ( \frac{n-2}{r} -\frac{1}{rf^2} +(n-1)r \right )
\frac{\partial f}{\partial r}\\
&\, -\frac{(n-2)}{r^2 f} \left ( f^2 -1 \right )\ ,\\
\end{split}
\end{equation}
We may write (\ref{eq3.9}) as
\begin{equation}
\label{eq3.10}
\begin{split}
\frac{\partial f}{\partial t} =&\, \Delta f -\frac{1}{f}
\left \vert d f \right \vert^2 + Y(f) -\frac{(n-2)}{r^2 f}
\left ( f^2 -1 \right )\\
Y:=&\, \left ( \frac{n-2}{r} -\frac{n}{rf^2} +(n-1)r \right )
\frac{\partial }{\partial r}.
\end{split}
\end{equation}
where, for any sufficiently differentiable function $u$ of $r$ alone (and possibly $t$), $\Delta$ denotes the scalar Laplacian $\Delta:=\nabla^i\nabla_i$ for the ansatz (\ref{eq3.3}), and so obeys
\begin{equation}
\label{eq3.11}
\Delta u(r) = \frac{1}{r^{n-1}f}\frac{\partial}{\partial r}
\left ( \frac{r^{n-1}}{f} \frac{\partial u}{\partial r} \right )
=\frac{1}{f^2} \frac{\partial^2 u}{\partial r^2}+\frac{(n-1)}{rf^2}
\frac{\partial u}{\partial r} - \frac{1}{f^3} \frac{\partial f}{\partial r}
\frac{\partial u}{\partial r}\ .
\end{equation}
We also employ the notation $\vert \omega \vert \equiv \vert \omega \vert_g :=\sqrt{g^{-1}(\omega,\omega)}$ for $\omega$ a one-form, so that for the radial one-form $df$ we have $\vert \omega \vert^2=\left ( \frac{1}{f}\frac{\partial f}{\partial r}\right )^2$.

Along the flow (\ref{eq3.3}), define two functions
\begin{eqnarray}
\label{eq3.12} \kappa &:=& \frac{1}{rf^3}\frac{\partial f}{\partial r}\ , \\
\label{eq3.13} \lambda &:=& \frac{1}{r^2}\left ( 1-\frac{1}{f^2} \right ) \ .
\end{eqnarray}
These are the sectional curvatures of ${\hat g}$, with $\kappa$ being the sectional curvature in $2$-planes containing $\frac{\partial}{\partial r}$ and $\lambda$ being the sectional curvature in $2$-planes orthogonal to $\frac{\partial}{\partial r}$. They pull back to the sectional curvatures along the normalized Ricci flow (\ref{eq3.1}). By asymptotic hyperbolicity the latter tend to $-1$ as $\rho\to\infty$. Thus so do $\kappa$ and $\lambda$ as $r\to\infty$. Observe that then $f\sim 1/r$ for large $r$. Also, $\kappa$ and $\lambda$ are related to each other by the Bianchi identity, which in this case can be read off from (\ref{eq3.12}) and (\ref{eq3.13}):
\begin{equation}
\label{eq3.14} r \frac{\partial \lambda}{\partial r} = 2(\kappa-\lambda)\ .
\end{equation}

%
\begin{comment}
By virtue of its construction, the function $f$ obeys (\ref{eq3.9}) and has fall-off behavior such that the functions $\kappa$ and $\lambda$ tend to $-1$ as $r\to\infty$. So, in a nutshell our method is to apply maximum principle arguments to (\ref{eq3.9}) (and other concomitant equations) to show that the fall-off behaviour, continuity, and differentiability are actually uniform in $t$ (at least for a large set of initial data). We can in turn use this to show long-time existence for solutions of (\ref{eq3.9}), hence for flowing metrics (\ref{eq3.3}) and, by pullback, for (\ref{eq3.1}).
\end{comment}
%

We summarize:

\begin{lemma}\label{lemma3.4}
Given a rotationally symmetric, smoothly conformally compactifiable, asymptotically hyperbolic metric $g_0$ with no closed minimal hypersurfaces, there is a $T>0$ such that the flow (\ref{eq3.1}) exists for $0\le t <T$, and is the pullback along $F_t^{-1}$ of the flow (\ref{eq3.3}). For each $t\in [0,T)$, the family of diffeomorphisms $F_t$ is smooth in $t$ and preserves $\sec \to -1$, $f\sim 1/r$, as $r\to\infty$. For each $t\in [0,T)$, the function $f$ defined by (\ref{eq3.4}) is smooth in $t$ and $r$ and obeys equation (\ref{eq3.9}).
\end{lemma}

\begin{comment}
In what follows, we will always assume smooth conformally compact asymptotically hyperbolic initial data so that $f(0,r)\to 0$, $\lambda(0,r)\to -1$, and $\kappa(0,r)\to -1$ as $r\to\infty$. In view of the above, we may then assume that $f(t,r)$ exists for $t\in [0,T)$ and that $f(t,r)\to 0$, $\lambda(t,r)\to -1$, and $\kappa(t,r)\to -1$ as $r\to \infty$, though the rates of convergence to these limits may be $T$-dependent. In the next two sections, we show that these rates are in fact independent of $T$.
\end{comment}

\subsection{Extended diffeomorphism}

\noindent As stated above $F_t$ and its inverse are only diffeomorphisms of $\mathcal{M}$.  However we will show that this diffeomorphism extends to a diffeomorphism of $M$. We prove this for thoroughness, but is not used in the sequel.

We first state a condition that allows us to check the preservation of conformal compactness.

\begin{lemma} \label{lemma3.2}
Consider a rotationally symmetric metric $g(t)$, $t\in [0,T)$, of the form (\ref{eq3.1}), with $g(0)=g_0$ and with $\rho$ the $g_0$-distance from the centre of symmetry. Then, for any $t\in [0,T)$, $g(t)$ is conformally compact if and only if any number of applications of the operator $e^{\rho} \partial_{\rho}$ to both $\phi_t^2(\rho)$ and $e^{-2\rho} \psi_t^2(\rho)$ remain bounded as $\rho \to \infty$ with $t$ fixed.
\end{lemma}

\begin{proof}
Recall that $\Omega = e^{-\rho}$ is a boundary defining function for $M$ the compact manifold with boundary that contains $\mathcal{M}$. The condition that $g$ is conformally compact means that $\Omega^2 g$ extends to a smooth metric on $M$. Thus, replacing the $\rho$ coordinate by $\Omega$ and multiplying by $\Omega^2$, we obtain
\begin{equation}
\label{eq3.15}
\Omega^2 g(t) = \phi_t^2(-\log \Omega) d \Omega^2 + \Omega^2 \psi_t^2(-\log \Omega) g\left (\bS^{n-1},{\rm can}\right ).
\end{equation}
Then $\phi_t^2(-\log \Omega)$ and $\Omega^2 \psi_t^2(-\log \Omega)$ are smooth on $M$; i.e., any number of $\Omega$-derivatives of these functions remains bounded as $\Omega \to 0$. Converting to $\rho$ we find $\partial_{\Omega} = - e^{\rho} \partial_{\rho}$, so that if $g$ is conformally compact, then for any $k \in \bN$, $(- e^{\rho} \partial_{\rho})^k (\phi_t^2(\rho))$ and $(- e^{\rho} \partial_{\rho})^k ( e^{-2\rho} \psi_t^2(\rho))$ remain bounded as $\rho \to \infty$. Conversely, if this condition holds for any $k$ as $\rho \to \infty$, then we see that any $\Omega$-derivative of $\phi_t^2(-\log \Omega)$ and $\Omega^2 \psi_t^2(-\log \Omega)$ remains bounded as $\Omega \to 0$. But then each $\Omega$-derivative of the compactified metric component satisfies a Lipschitz condition which allows us to extend each derivative to the manifold with boundary. So $g$ is conformally compact.
\end{proof}

\begin{lemma} \label{lemma3.3}
The diffeomorphism $F_t: \mathcal{M} \to \mathcal{M}: F_t(\rho,\theta^A)\mapsto (r,\theta^A)$ defined above extends to a diffeomorphism $F_t: M \to M$ of manifolds with boundary. Consequently $\hat{g} = \left (F_t^{-1}\right )^* g$ remains conformally compact.
\end{lemma}
\begin{proof}
The diffeomorphism $F_t$ is the identity in the angular directions and so we need only consider the radial direction. In the interior, the diffeomorphism $(0,\infty)_{\rho}$ to $(0,\infty)_r$ is given by $r = \psi_t(\rho)$. Rescaling, we set $\Omega = e^{-\rho}$ as before and let $s = \frac{1}{r}$. We will be interested in the composition $s = s(\Omega)$, initially defined for $\Omega > 0$ (and obviously $r>0$) by
\begin{equation}
\label{eq3.16}
s = \frac{1}{r} = \frac{1}{\psi_t(\rho)} = \frac{1}{\psi_t(- \log \Omega)}\ ,
\end{equation}
and its inverse $\Omega = \Omega(s)$, defined for $s > 0$ by
\begin{equation}
\label{eq3.17}
\Omega = e^{-\psi_t^{-1}(1/s)}\ .
\end{equation}
To show that these diffeomophisms extend to $\Omega = 0$ and  $s = 0$, we convert the requisite $\Omega$ and $s$ derivatives to $\rho$ and then argue that the control on $\phi_t$ and $\psi_t$ coming from conformal compactness via Lemma \ref{lemma3.2} gives a bound on the $\Omega$ and $s$ derivatives down to zero. Then these derivatives are Lipschitz and extend to the boundary, proving that $F_t$ is a diffeomorphism of manifolds with boundary.

%Let us introduce some convenient notation that will be used in this proof only.
Let $\sE$ denote the algebra of functions such that any number of $e^{\rho} \partial_{\rho}$ derivatives remain bounded as $\rho \to \infty$.
%By hypothesis, Lemma \ref{lemma3.2} implies that $e^{-\rho} \psi(\rho) \in \sE$, and %a simple computation shows that
%$e^{-\rho} \in \sE$.
Obviously, $e^{-\rho} \in \sE$ and, by Lemma \ref{lemma3.2}, $e^{-\rho} \psi(\rho) \in \sE$. Further, $\partial_{\rho} \psi_t - \psi_t = e^{\rho} \partial_{\rho} ( e^{-\rho} \psi_t ) \in \sE$, so that $e^{-\rho} \partial_{\rho} \psi_t \in \sE$ as well. Consider the function $ s = \frac{1}{\psi(- \log \Omega)}$. We find
\begin{equation}
\label{eq3.18}
\frac{ds}{d\Omega} = \frac{1}{\Omega} \frac{\partial_{\rho}\psi_t(-\log \Omega)}{\psi_t^2(-\log \Omega)} = \frac{ \partial_{\rho}\psi_t(\rho)}{ e^{-\rho} \psi_t^2(\rho)} = \frac{ e^{-\rho}\partial_{\rho} \psi_t(\rho)}{ e^{-2\rho} \psi_t^2(\rho)} \in \sE
\end{equation}
%lies in $\sE$
and so $\frac{ds}{d\Omega}$ is bounded as $\rho \to \infty$ (equivalently $\Omega \to 0$). Noting that $\frac{\partial}{\partial \Omega} = - e^{\rho} \frac{\partial}{\partial \rho}: \sE \to \sE$, we conclude that the forward map extends to the boundary.

We now consider the inverse map. We may differentiate the equation $\psi_t^{-1} ( \psi_t( \rho ) ) = \rho$, to obtain $(\psi_t^{-1})'(r) = \frac{1}{\partial_{\rho}\psi_t(\rho)}$.\footnote
{Because of the absence of closed minimal hyperspheres, cf. Remark \ref{remark1.3}, this derivative exists everywhere (excluding the origin, since $r$ is a polar coordinate), but here we only need it to exist on a collar neighbourhood of infinity.}
Further, $\frac{\partial}{\partial s} = -\frac{\psi_t^2(\rho)}{\partial_{\rho}\psi_t(\rho)} \frac{\partial}{\partial \rho}$. Differentiating (\ref{eq3.17}), we obtain
\begin{equation}
\label{eq3.19}
\begin{split}
\frac{d \Omega}{ds} =&\, \frac{1}{s^2} e^{-\psi_t^{-1}(1/s)}\left (\psi_t^{-1}\right )' ( 1/s) =
\frac{e^{-\rho} \psi_t^2(\rho)}{ \partial_{\rho} \psi_t(\rho)}\\
=&\, (e^{-\rho} \psi_t(\rho)) \left( \frac{ e^{-\rho} \psi_t(\rho)}{ e^{-\rho} \partial_{\rho}\psi_t(\rho)} \right)\ ,
\end{split}
\end{equation}
so that $\frac{d \Omega}{ds} \in \sE$, which in turn shows that the $s$-derivative of $\Omega$ remains bounded as $\rho \to \infty$ ($s\to 0$). This same argument shows that $\frac{\partial}{\partial s} = -\frac{ e^{-2\rho} \psi_t^2(\rho)}{e^{-\rho}  \partial_{\rho}\psi_t(\rho)} e^{\rho} \frac{\partial}{\partial \rho}$, so that in fact $\frac{d}{ds}: \sE \to \sE$. So $F_t$ extends to a diffeomorphism of manifolds with boundary.

Finally, since $s$ and $\Omega$ pull back to each other under the rescaled diffeomorphism, a short computation shows that $\hat{g}$ is conformally compact.
\end{proof}

\section{Uniform $C^0$ Control of $f$}
\setcounter{equation}{0}

\subsection{Maximum principle}

\noindent As discussed above, the main issue we have to deal with is the possibility that the metric and curvature quantities that we wish to study become unbounded at the singular point at the origin of symmetry. We therefore would like a version of the usual maximum principle which is adapted to this situation.

We denote $\Omega_T=(0,T)\times (0,\infty)$, ${\bar \Omega}_T:=[0,T]\times [0,\infty)$, and ${\hat \Omega}_T :=(0,T]\times(0,\infty)$. Then

\begin{proposition}\label{proposition4.1}
Let $u\in C^{0,0}({\bar \Omega}_T)\cap C^{1,2}(\Omega_T)$ satisfy
\begin{equation}
\label{eq4.1} 0\le L(u) := a(t,r)\frac{\partial^2 u}{\partial r^2}+b(t,r)\frac{\partial u}{\partial r}+c(t,r)u
-\frac{\partial u}{\partial t}\end{equation}
with $0<C_1\le a(t,r)\le C_2$, on $\Omega_T$, with $a,b\in C^{(0,0)}(\Omega_T)$ and $c\in C^{0,0}({\hat \Omega}_T)$. Furthermore, assume that $u(t,0)= u_0$, that $u(t,r)\to u_{\infty}\in{\mathbb R}$ as $r\to \infty$, and that $\max \{ u_0,u_{\infty} \}\ge 0$. Either
\begin{enumerate}
\item $\sup_{{\bar \Omega}_T}u=\max \{ u_0,u_{\infty} \}$, or
\item $\sup_{{\bar \Omega}_T}u$ is a maximum realized at $(0,r^*)$ for some $r^*\in (0,\infty)$, or
\item $\sup_{{\bar \Omega}_T}u$ is a maximum realized at some point $(t^*,r^*)\in (0,T]\times (0,\infty)$, and $c(t^*,r^*)\ge 0$.
\end{enumerate}
\end{proposition}

We have extended the continuity condition on $c$ to include $t=T$ even though (\ref{eq4.1}) holds only in $\Omega_T$ since otherwise the third alternative in the theorem could fail to hold as stated. This condition will always hold in our applications.

\begin{proof} If neither possibility 1 nor possibility 2 above applies, then $u$ achieves a positive maximum at some point $(t^*,r^*)\in (0,T]\times (0,\infty)$. Then $v(t,r)=u(t,r)-u(t^*,r^*)$ has maximum $v(t^*,r^*)=0$ in $(0,T]\times (0,\infty)$ and obeys $L(v)\ge -c(t,r)u(t^*,r^*)$. Consider a small rectangle $R=[t^*-\epsilon,t^*]\times[r^*-\epsilon,r^*+\epsilon]\subset (0,T]\times (0,\infty)$ about $(t^*,r^*)$, such that $u(t,r)>0$ in $R$. If $v(t,r)$ is not identically zero in $R$, it is necessary that $c(t,r)>0$ somewhere in $R$ so as not to contradict the strong maximum principle \cite[Theorem 2.7]{Lieberman}, and by repeating the argument as we take $\epsilon\to 0$, we obtain a sequence of points $(t_i,r_i)$ converging to $(t^*,r^*)$ such that $c(t_i,r_i)>0$. Thus $c(t^*,r^*)\ge 0$. Alternatively, if $v(t,r)\equiv 0$ in $R$, then $u(t,r)\equiv u(t^*,r^*)$ and therefore $u(t,r)>0$ in $R$, and also $L(v)=L(0)=0$ in $R$. But since $L(v)\ge -c(t,r)u(t^*,r^*)$, then $c(t,r)\ge 0$ throughout $R$, and so again $c(t^*,r^*)\ge 0$. \end{proof}

Though the inequality (\ref{eq4.1}) applies when right-hand side is linear in $u$, we will apply the proposition when the right side of the equation is only quasi-linear. Thus the coefficients $a$, $b$, and $c$ may arise from compositions of continuous functions; e.g., $b(t,r) = \beta(t,r;p,q)$ with $p=u(t,r)$ and $q=\frac{\partial u}{\partial r}(t,r)$. Alternatively, in the proof one instead could appeal to the strong maximum principle for quasi-linear equations \cite[Theorem 2]{Kusano}.  We conclude this subsection with the corresponding parabolic minimum principle.

\begin{corollary}\label{corollary4.2}
Let $u\in C^{0,0}({\bar \Omega}_T)\cap C^{1,2}(\Omega_T)$ satisfy
\begin{equation}
\label{eq4.2} 0\ge L(u) := a(t,r)\frac{\partial^2 u}{\partial r^2}+b(t,r)\frac{\partial u}{\partial r}+c(t,r)u
-\frac{\partial u}{\partial t}\end{equation}
with $0<C_1\le a(t,r)\le C_2$, on $\Omega_T$, with $a,b\in C^{(0,0)}(\Omega_T)$ and $c\in C^{0,0}({\hat \Omega}_T)$. Furthermore, assume that $u(t,0)= u_0$, that $u(t,r)\to u_{\infty}\in{\mathbb R}$ as $r\to \infty$, and that $\min \{ u_0,u_{\infty} \}\le 0$. Either
\begin{enumerate}
\item $\inf_{{\bar \Omega}_T}u=\min \{ u_0,u_{\infty} \}$, or
\item $\inf_{{\bar \Omega}_T}u$ is a minimum realized at $(0,r^*)$ for some $r^*\in (0,\infty)$, or
\item $\inf_{{\bar \Omega}_T}u$ is a minimum realized at some point $(t^*,r^*)\in (0,T]\times (0,\infty)$, and $c(t^*,r^*)\ge 0$.
\end{enumerate}
\end{corollary}

\begin{proof} Replace $u$ by $-u$ in the proof of Proposition \ref{proposition4.1}. Note that the sign condition on $c(t^*,r^*)$ in the third alternative above is \emph{not} reversed.
\end{proof}

\subsection{No neckpinches form}

\noindent Letting $w=f^2$, we obtain from (\ref{eq3.9}) the evolution equation
\begin{equation}
\label{eq4.3} \frac{\partial w}{\partial t}
= \Delta w -\frac{1}{w} \left \vert dw
\right \vert^2+Y(w) -\frac{2(n-2)}{r^2} \left ( w-1 \right )\ .
\end{equation}
By expanding terms and dropping certain terms of definite sign, we obtain
\begin{equation}
\label{eq4.4} \frac{\partial w}{\partial t}
\le \frac{1}{w}\frac{\partial^2 w}{\partial r^2}+\left ( \frac{n-2}{r}-\frac{n}{rw}+(n-1)r \right )
\frac{\partial w}{\partial r} -\frac{2(n-2)}{r^2} \left ( w-1 \right )\ .
\end{equation}

There is a singular point at $r=0$ but the divisions by $w$ cause no problem since, by long-time existence, $w$ does not become zero for $t\le T$ (though it tends to zero as $r\to\infty$). We may now apply the maximum principle.

\begin{lemma}\label{lemma4.3}
Suppose that the flow (\ref{eq3.1}) obeys (\ref{eq1.1}) for $t\in [0,T_{\max})$ where $T_{\max}\in (0,\infty]$ is the maximal time of existence. Furthermore suppose that the initial metric $g_0=g(0,\cdot)$ does not admit a minimal hypersphere.\footnote
{We remind the reader that the absence of initial minimal hyperspheres follows from our assumption $\lambda(0)\le 0$, as explained in Remark \ref{remark1.3}. As well, from an elementary argument presented as part of the proof of Remark \ref{remark1.3}, we see that the absence of minimal hyperspheres at any time $t$ precludes the existence of any closed minimal hypersurface at time $t$.}
Then $f$ is uniformly bounded for $t\in [0,T_{\max}]$, no minimal hypersphere forms, and the coordinate transformation (\ref{eq3.4}) exists (and is invertible) for $t\in [0,T_{\max})$.
\end{lemma}

\begin{proof} By (\ref{eq3.4}) and the condition that no minimal hypersphere is present initially, we then have that $\sup_r f(0,r)\le C$ for some $C\ge 1$ (since by smoothness at $t=0$, we have $f(0,0)=1$). Now let $T_0>0$ be the first time at which a minimal hypersphere forms along the flow (\ref{eq3.1}), and let $T_1:=\min \{ T_0, T_{\max} \}$. Consider expression (\ref{eq4.4}) on $(0,T_1]\times (0,\infty)$, so as to avoid the singularity at $r=0$. This expression has the same form as (\ref{eq4.1}) with $u=w-1$, $a=\frac{1}{w}$, $b=\frac{n-2}{r}-\frac{1}{rw}+(n-1)r$, and $c=-\frac{2(n-2)}{r^2}$. As well, $w=1$ at $r=0$ and $w\to 0$ as $r\to\infty$, so Proposition \ref{proposition4.1} applies and asserts that $w(t,r)$ is bounded above by $C^2$. Thus $f$ is uniformly bounded on $(0,T_1]\times (0,\infty)$. But by (\ref{eq3.5}), a compact minimal hypersurface cannot form unless $f$ diverges, and so $T_1=T_{\max}$.
\end{proof}

We note in passing that this does not forbid compact minimal hypersurfaces from forming if some are already present initially, since then the assumption $\sup_{x\in{\cal M}} f(0,x)=C$ does not hold and our coordinate system fails to cover the initial manifold.

\begin{corollary}[Preservation of nonpositive $\lambda$] \label{corollary4.4}
If $\sup \lambda(0,r)\le 0$ for $r\in[0,\infty)$, then $\sup_r \lambda(t,r)\le 0$ for $(t,r)\in [0,T_{\max}]\times [0,\infty)$.
\end{corollary}

\begin{proof} We first observe that if $\sup \lambda(0,r)\le 0$, then $\sup f(0,r)\le 1$ by equation \eqref{eq3.13}, and then in fact we have $\max f(0,r)=1$ since $f(0,0)=1$ by smoothness. But by Lemma \ref{lemma4.3}, if $\max f(0,r)= 1$, then $\max f(t,r)= 1$ for all $(t,r)\in [0,T_{\max}]\times[0,\infty)$, and then it follows from equation \eqref{eq3.13} that $\sup\lambda(t,r)\le 0$ all along the flow.
\end{proof}

In the next section we will prove decay estimates for $\lambda$.

\section{Evolution of $\lambda$}
\setcounter{equation}{0}

\subsection{Evolution equation for $\lambda$}

\noindent From (\ref{eq4.3}), we can write
\begin{equation}
\label{eq5.1}
\begin{split}
\frac{\partial}{\partial t} \left ( 1-\frac{1}{f^2} \right )
=&\ \Delta \left ( 1-\frac{1}{f^2} \right )+f^2 \left \vert d \left ( 1-\frac{1}{f^2} \right ) \right \vert^2 +Y \left ( 1-\frac{1}{f^2} \right )\\
&\ +\frac{2(n-2)}{r^2}\left ( 1-\frac{1}{f^2} \right )^2
-\frac{2(n-2)}{r^2}\left ( 1-\frac{1}{f^2} \right )\ .
\end{split}
\end{equation}
We also need several identities, the first of which is
\begin{equation}
\label{eq5.2}
\begin{split}
\Delta r =&\,\, \frac{1}{r^{n-1}f}\frac{\partial}{\partial r} \left (
\frac{r^{n-1}}{f} \right ) = \frac{n-1}{rf^2}
-\frac{1}{f^3} \frac{\partial f}{\partial r}\\
&=\,\, -(n-1)\frac{1}{r}\left ( 1 -\frac{1}{f^2} \right ) +\frac{n-1}{r}
-r\kappa\\
&=\,\, -nr\lambda -r\left ( \kappa -\lambda \right )+\frac{n-1}{r}\\
&=\,\, -nr\lambda-\frac12 r^2 \frac{\partial \lambda}{\partial r} +\frac{n-1}{r}\ .
\end{split}
\end{equation}
By straightforward manipulations, some of which use (\ref{eq3.13}), we then further obtain the identities
\begin{eqnarray}
\label{eq5.3} \frac{1}{r^2} d \left ( 1-\frac{1}{f^2} \right ) &=&
d \lambda +2\lambda \frac{dr}{r}\ , \\
\label{eq5.4} \frac{1}{r^2}f^2 \left \vert d \left ( 1-\frac{1}{f^2} \right )
\right \vert^2 &=& r^2 f^2 \vert d\lambda \vert^2 +4r\lambda f^2 dr
\cdot d\lambda +4\lambda^2\ ,\\
\label{eq5.5} \frac{1}{r^2}\Delta \left ( 1-\frac{1}{f^2} \right ) &=& \Delta
\lambda +\left ( 4-r^2f^2\lambda \right )\frac{dr}{r}\cdot d\lambda
-2(n+1)\lambda^2+\frac{2n}{r^2}\lambda\\
\label{eq5.6} \frac{1}{r^2} Y\left ( 1-\frac{1}{f^2} \right ) &=&
\left ( nr\lambda -\frac{2}{r} +(n-1)r \right ) f^2 dr \cdot d\lambda
\nonumber\\
&& +2n\lambda^2+2\left ( n-1 -\frac{2}{r^2} \right ) \lambda \ ,
\end{eqnarray}
where $dr\cdot d\lambda:=g^{-1}(dr,d\lambda)\equiv \frac{1}{f^2} \frac{\partial \lambda}{dr}$. Multiplying (\ref{eq5.1}) by $1/r^2$ and using the above equations, we obtain the desired evolution equation
\begin{equation}
\label{eq5.7} \frac{\partial \lambda}{\partial t}
= \Delta \lambda + Z\cdot d\lambda +r^2f^2\left \vert d\lambda \right \vert^2
+2(n-1)\lambda (\lambda +1)\ ,
\end{equation}
where we define
\begin{equation}
\label{eq5.8}
Z:= \left [ \frac{(n+1)}{r}-\frac{(n-1)}{rf^2}+(n-1)r \right ] f^2 dr
= \left [ \frac{2}{r} +(n-1)r(\lambda+1) \right ] f^2dr\ .
\end{equation}

\subsection{Decay of $\lambda$}

\noindent It is convenient to define $\Lambda(t,r)$ by
\begin{equation}
\label{eq5.9} \lambda(t,r)=:-1+\Lambda(t,r)e^{-2(n-1)\delta^2t}\ ,
\end{equation}
where $\delta$ will be assigned values $\delta=0$ and $\delta=1$ when needed below. Then
\begin{equation}
\label{eq5.10} \frac{\partial \Lambda}{\partial t}
= \Delta \Lambda + Z\cdot d\Lambda +e^{-2(n-1)\delta^2 t}r^2f^2\left \vert d\Lambda \right \vert^2+2(n-1)\left ( \lambda+\delta^2\right )\Lambda\ .
\end{equation}
The idea now is to show that $\Lambda$ is bounded uniformly (and, even better, converging) in $t$. A lower bound seems obvious from inspection of (\ref{eq5.10}). However, the norm of $Z$ diverges at $r=0$ so when (\ref{eq5.10}) is expressed in any coordinate system that covers the origin, the result is a singular PDE. This was also the case with $Y$ in (\ref{eq4.3}), and there the strategy was to work on the domain $[0,T]\times [\epsilon,1/\epsilon]$, $T < T_{\max}$, and then take $\epsilon\searrow 0$, using that $w\to 1$ as $r\searrow 0$. We will follow the same strategy here, but this alone is insufficient because we have no uniform control of $\Lambda$ at the origin. Instead, we borrow a trick from \cite{OW} and study the functions
\begin{equation}
\label{eq5.11} \Lambda_{\alpha}:=\frac{r^{\alpha}}{(r^{\alpha} + \alpha)}\Lambda\ , \ \alpha\in(0,1)\ ,
\end{equation}
subsequently taking $\alpha\searrow 0$. The factor $\frac{r^{\alpha}}{r^{\alpha}+\alpha}$ is chosen because we will have $\Lambda_{\alpha}(t,0)= 0$, while maintaining $\Lambda_{\alpha}(t,r)\sim\Lambda\to 0$ as $r\to \infty$, and we will have $\lim_{\alpha}\Lambda_{\alpha}=\Lambda$ (non-uniformly).

We again need some identities. First,
\begin{equation}
\label{eq5.12} \Lambda = \left ( 1 + \frac{\alpha}{r^{\alpha}} \right ) \Lambda_{\alpha}\ ,
\end{equation}
so
\begin{equation}
\label{eq5.13} d\Lambda = \left ( 1+\frac{\alpha}{r^{\alpha}} \right )
d\Lambda_{\alpha} -\frac{\alpha^2}{r^{\alpha+1}} \Lambda_{\alpha} dr\ .
\end{equation}
Omitting the exponential factor for brevity, this allows us to write out the term in (\ref{eq5.10}) that is quadratic in $d\Lambda$ as
\begin{equation}
\label{eq5.14}
\begin{split}
r^2f^2 |d\Lambda|^2 =&\, r^2f^2 \left ( 1+\frac{\alpha}{r^{\alpha}} \right )^2
\left \vert d\Lambda_{\alpha} \right \vert^2 - 2\alpha^2 r^{1-\alpha}f^2\Lambda_{\alpha} \left (
1+\frac{\alpha}{r^{\alpha}} \right )dr\cdot d\Lambda_{\alpha}+\frac{\alpha^4}{r^{2\alpha}}\Lambda_{\alpha}^2\\
=&\, r^2f^2 \left ( 1+\frac{\alpha}{r^{\alpha}} \right )^2
\left \vert d\Lambda_{\alpha} \right \vert^2 - 2\alpha^2 r^{1-\alpha}f^2\Lambda_{\alpha} \left (
1+\frac{\alpha}{r^{\alpha}} \right )dr\cdot d\Lambda_{\alpha}\\
&\, +\frac{\alpha^4e^{2(n-1)\delta^2t}}{r^{\alpha}(r^{\alpha}+\alpha)}(\lambda+1) \Lambda_{\alpha}\ ,
\end{split}
\end{equation}
using $f^2|dr|^2=1$. In the last line, we have rewritten the quadratic combination $\Lambda_{\alpha}^2$ in terms of $(\lambda+1)\Lambda_{\alpha}$. This is for convenience in what follows (though the $\Lambda_{\alpha}^2$ form also plays a role below), where by imposing sign conditions on $\lambda$ we will convert the differential equation for $\Lambda_{\alpha}$ into an inequality with linear nonderivative terms (\emph{cf.} the nonderivative term in equation (\ref{eq5.10})).

The $Z\cdot d\Lambda$ term is
\begin{equation}
\label{eq5.15}
\begin{split}
Z\cdot d\Lambda =&\,\, \left ( 1+\frac{\alpha}{r^{\alpha}} \right ) Z\cdot d\Lambda_{\alpha}
-\frac{\alpha^2}{r^{\alpha+1}} \Lambda_{\alpha} \left ( \frac{2}{r}
+(n-1)r(\lambda+1) \right )\\
=&\,\, \left ( 1+\frac{\alpha}{r^{\alpha}} \right ) Z\cdot d\Lambda_{\alpha}-\frac{2\alpha^2}{r^{\alpha+2}}
\Lambda_{\alpha} -\frac{\alpha^2(n-1)}{r^{\alpha}}
\left ( \lambda+1\right ) \Lambda_{\alpha}\ .
\end{split}
\end{equation}

Next we have to work out the relationship between $\Delta \Lambda$
and $\Delta \Lambda_{\alpha}$. Taking the divergence of (\ref{eq5.13}), we
have
\begin{equation}
\label{eq5.16}
\begin{split}
\Delta \Lambda =&\,\, \left ( 1 + \frac{\alpha}{r^{\alpha}} \right ) \Delta \Lambda_{\alpha}
-\frac{2\alpha^2}{r^{\alpha+1}f^2}\frac{\partial \Lambda_{\alpha}}{\partial r} +\frac{\alpha^2(\alpha+1)}{r^{\alpha+2}f^2}\Lambda_{\alpha}
-\frac{\alpha^2}{r^{\alpha+1}}\Lambda_{\alpha} \Delta r\\
=&\,\, \left ( 1 + \frac{\alpha}{r^{\alpha}} \right ) \Delta \Lambda_{\alpha}
-\frac{2\alpha^2}{r^{\alpha+1}f^2}\frac{\partial \Lambda_{\alpha}}{\partial r} +\frac{\alpha^2(\alpha+1)}{r^{\alpha+2}}\Lambda_{\alpha}
-\frac{\alpha^2(\alpha+1)}{r^{\alpha}}\Lambda_{\alpha}\lambda\\
&\,\, +\frac{\alpha^2}{r^{\alpha+1}}\Lambda_{\alpha} \left [ nr\lambda+\frac12 r^2 \frac{\partial \lambda}{\partial r} -\frac{n-1}{r}\right ]\ ,
\end{split}
\end{equation}
where we replaced factors of $\frac{1}{r^2f^2}$ using $\frac{1}{r^2f^2}=\frac{1}{r^2}-\lambda$ and we used (\ref{eq5.2}) to replace $\Delta r$ by
\begin{equation}
\label{eq5.17}
\Delta r =\, -nr\lambda-\frac12 r^2 \frac{\partial \lambda}{\partial r}  +\frac{n-1}{r}\ .
\end{equation}
Next, we need to replace $\frac{\partial \lambda}{\partial r}$ in favour of $\frac{\partial \Lambda_{\alpha}}{\partial r}$, using \eqref{eq5.9} and \eqref{eq5.13}. After considerable manipulation, the result can be written as
\begin{equation}
\label{eq5.18}
\begin{split}
\Delta \Lambda =&\,\, \left ( 1 + \frac{\alpha}{r^{\alpha}} \right ) \Delta \Lambda_{\alpha}
+\frac{\alpha^2}{r^{\alpha}}\left [ \frac12 r\left ( \lambda+1 \right )-\frac{2}{rf^2}\right ] \frac{\partial\Lambda_{\alpha}}{\partial r}\\
&\,\, -\frac{\alpha^2}{r^{\alpha}}\left [ \frac{n-\alpha-2}{r^2}+\frac{\alpha^2}{2(r^{\alpha}+\alpha)}\right ]\Lambda_{\alpha}\\
&\,\, +\frac{\alpha^2}{r^{\alpha}}\left [ n-\alpha-1 -\frac{\alpha^2}{2(r^{\alpha}+\alpha)}\right ]
\lambda\Lambda_{\alpha} \ .
\end{split}
\end{equation}

Plugging all this into (\ref{eq5.10}) and organizing terms, we get the evolution equation for $\Lambda_{\alpha}$:
\begin{equation}
\label{eq5.19}
\begin{split}
\frac{\partial \Lambda_{\alpha}}{\partial t} =&\, \Delta \Lambda_{\alpha} +Z\cdot d\Lambda_{\alpha}
+e^{-2(n-1)\delta^2t}r^2\left ( 1+\frac{\alpha}{r^{\alpha}} \right )\left ( \frac{\partial \Lambda_{\alpha}}{\partial r}\right )^2\\
&\, -\frac{\alpha^2}{(r^{\alpha}+\alpha)}\left [ \frac32 r\left (\lambda +1\right ) +\frac{2}{rf^2}\right ] \frac{\partial \Lambda_{\alpha}}{\partial r}\\
&\, +2(n-1)\left ( \delta^2+\lambda\right ) \Lambda_{\alpha}\\
&\, -\frac{\alpha^2}{(r^{\alpha}+\alpha)}\left [n-1 -\frac{\alpha^2}{2(r^{\alpha}+\alpha)}+ \frac{(n-\alpha)}{r^2}\right ] \Lambda_{\alpha}\\
&\, +\frac{\alpha^3}{(r^{\alpha}+\alpha)}\left [ \frac{\alpha}{2(r^{\alpha}+\alpha)}-1\right ] \lambda\Lambda_{\alpha}\ .
\end{split}
\end{equation}
Comparing to (\ref{eq4.2}), we have
\begin{equation}
\label{eq5.20}
\begin{split}
c(t,r)=&\, 2(n-1)\left ( \delta^2+\lambda\right ) \\
&\, -\frac{\alpha^2}{(r^{\alpha}+\alpha)}\left [n-1 -\frac{\alpha^2}{2(r^{\alpha}+\alpha)}+ \frac{(n-\alpha)}{r^2}\right ] \\
&\, +\frac{\alpha^3}{(r^{\alpha}+\alpha)}\left [ \frac{\alpha}{2(r^{\alpha}+\alpha)}-1\right ] \lambda \ .
\end{split}
\end{equation}

\begin{proposition}\label{proposition5.1}
%\begin{equation}
$ \lambda(t,r)\ge -1+e^{-2(n-1)t}\inf_{r\in[0,\infty)} \left ( \lambda(0,r)+1\right )$ for all $(t,r) \in [0,T_{\max})\times [0,\infty)$.
%\end{equation}
\end{proposition}

\begin{proof}
Choose some $T<T_{\max}$. If $\Lambda_{\alpha}\ge 0$ for all $(t,r)\in [0,T]\times [0,\infty)$ then $\lambda+1\ge 0$ and then $\inf_{r\in[0,\infty)} \left ( \lambda(0,r)+1\right )=\lim_{r\to\infty} \left ( \lambda(0,r)+1\right )=0$, so the proposition is true.

Otherwise, $\Lambda_{\alpha}$ has a negative minimum on $[0,T]\times [0,\infty)$. In this case, consider equations (\ref{eq5.19}, \ref{eq5.20}) but set $\delta=1$ so that
\begin{equation}
\label{eq5.21}
\lambda=-1+e^{-2(n-1)t}\left ( 1+\frac{\alpha}{r^{\alpha}}\right )\Lambda_{\alpha}\ .
\end{equation}
Now replace the $\lambda$ factors in (\ref{eq5.20}) using (\ref{eq5.21}) and of course set $\delta=1$. This yields
\begin{equation}
\begin{split}
\label{eq5.22} c(t,r) = &\,-\frac{\alpha^2}{(r^{\alpha}+\alpha)}\left [n-1 -\alpha+ \frac{(n-\alpha)}{r^2}\right ] \\
&\, +e^{-2(n-1) t} \left [ 2(n-1) +\frac{\alpha\left ( 2(n-1)-\alpha^2\right )}{r^{\alpha}} +\frac{\alpha^4}{2r^{\alpha}(r^{\alpha}+\alpha)}\right ]\Lambda_{\alpha}\ .
\end{split}
\end{equation}
Take $\alpha>0$ small enough; indeed, $\alpha\in (0,1)$ will do. Now, if $\Lambda_{\alpha}$ achieves a negative minimum at $(t^*,r^*)$ in $(0,T]\times(0,\infty)$, then from the right-hand side of (\ref{eq5.22}) we obtain
\begin{equation}
\label{eq5.23}
c(t^*,r^*)<0\ .
\end{equation}
But then Corollary \ref{corollary4.2} would imply that $t^*=0$.

Thus in either case we may write $\Lambda_{\alpha}(t,r)\ge \inf_{r\in[0,\infty)}\Lambda_{\alpha}(0,r) = \inf_{r\in[0,\infty)} \frac{r^{\alpha}}{r^{\alpha}+\alpha} \Lambda(0,r)\ge \inf_{r\in[0,\infty)}\Lambda(0,r)$, where the last inequality holds because the infimum is not positive. This  yields $\Lambda_{\alpha}(t,r) \ge \inf_{r\in[0,\infty)} \left ( \lambda(0,r)+1\right )$ and since this bound is independent of $\alpha$, we may take $\alpha\searrow 0$, yielding $\Lambda(t,r)\ge \inf_{r\in[0,\infty)} \left ( \lambda(0,r)+1\right )$. Then the proposition follows from (\ref{eq5.21}) and from the fact that the bound does not depend on $T$ so we can take $T\nearrow T_{\max}$. \end{proof}

On the other hand, by selecting other values of $\delta$ we can bound $\lambda$ from above. For example, setting $\delta=0$ in (\ref{eq5.19}, \ref{eq5.20}) yields
\begin{equation}
\label{eq5.24}
\begin{split}
c(t,r)=&\, 2(n-1)\lambda -\frac{\alpha^2}{(r^{\alpha}+\alpha)}\left [n-1 -\frac{\alpha^2}{2(r^{\alpha}+\alpha)}+ \frac{(n-\alpha)}{r^2}\right ] \\
&\, +\frac{\alpha^3}{(r^{\alpha}+\alpha)}\left [\frac{\alpha}{2(r^{\alpha}+\alpha)}-1\right ]\lambda\ ,
\end{split}
\end{equation}
with
\begin{equation}
\label{eq5.25}
\lambda=-1+\left ( 1+\frac{\alpha}{r^{\alpha}}\right )\Lambda_{\alpha}\ .
\end{equation}
Then we obtain

\begin{proposition}\label{proposition5.2}
Assume that $\sup_{r\in [0,\infty)} \lambda(0,r)\le -1$. Then $\lambda(t,r)\le -1$ for all $(t,r)
\in [0,T_{\max})\times [0,\infty)$.
\end{proposition}

\begin{proof} Assume to the contrary that $\sup_r \lambda(0,r)\le -1$ but $\lambda(t,r)>-1$ for some $(t,r)$ with $t>0$. Then there is some $T<T_{\max}$ such that $\lambda<0$ for all $t\in [0,T]$ but $\lambda$ is somewhere greater than $-1$. Since $\lambda+1$ and $\Lambda_{\alpha}$ have the same sign, then $\Lambda_{\alpha}>0$ somewhere. But $\Lambda_{\alpha}(0,r)\le 0$, $\Lambda_{\alpha}(t,0)=0$, and $\Lambda_{\alpha}(t,r)\to 0$ as $r\to\infty$, so by Proposition \ref{proposition4.1} we must conclude that $c(t^*,r^*)>0$ at some $(t^*,r^*) \in (0,T]\times (0,\infty)$.

Nonetheless, $\lambda<0$ on $[0,T]\times [0,\infty)$ and so, for $\alpha>0$ chosen to be small enough, then from (\ref{eq5.24}) we see that $c(t,r)$ is strictly negative in $(0,T]\times (0,\infty)$, a contradiction. Thus we conclude that $\lambda(t,r)\le -1$ for all $(t,r)\in [0,T]\times (0,\infty)$.
\end{proof}

We may in fact weaken our assumptions to allow $\lambda(0,r)<0$. We have seen that the bound from below is exponential and decays to $-1$ as $e^{-2(n-1)t}$. The decay from above is also exponential, but potentially slower.

\begin{proposition}\label{proposition5.3}
Assume that $\lambda(0,r)< 0$ for all $r\in [0,\infty)$. Then there is a constant $\beta>0$ depending only on the initial metric $g(0,\cdot) = g_0$ and $n$ such that $\lambda(t,r)< -1+ e^{-\beta t}$ for all $(t,r) \in [0,T_{\max})\times [0,\infty)$.
\end{proposition}

\begin{proof}
Since Proposition \ref{proposition5.1} deals with the bound from below, we restrict attention to bounding $\lambda(t,r)$ from above. Thus assume that $\lambda(0,r)>-1$ for some $r\in [0,\infty)$. Since $\lambda(0,r)<0$ for all $r\in [0,\infty)$ and tends to $-1$ as $r\to\infty$, then $\max_r \lambda(0,r)$ exists and lies in the interval $(-1,0)$.

Choose any $\delta\in (0,1)$ such that $\max_r \lambda(0,r)<-\delta^2$. Let $T_{\delta}\in (0,\infty]$ be the first time such that $\delta^2+\lambda(T_{\delta},r)=0$ for some $r\in [0,\infty)$; if there is no such time then $T_{\delta}:=T_{\rm max}$ (possibly infinite). But by (\ref{eq5.20}) and for $0<\alpha<\epsilon$ with $\epsilon$ small enough, we have
\begin{equation}
\label{eq5.26}
c(t,r)<2(n-1)\left ( \delta^2+\lambda \right ) \le 0
\end{equation}
for $t\in [0,T_{\delta}]$. Thus, by Proposition \ref{proposition4.1}, for each $\alpha\in (0,\epsilon)$, the maximum of $\Lambda_{\alpha}$ on $[0,T_{\delta}]\times [0,\infty)$ occurs at $t=0$. Then for $0\le t\le T_{\delta}$ and each $\alpha\in (0,\epsilon)$ we have $\Lambda_{\alpha}(t,r)\le \sup_{r\in[0,\infty)}\Lambda_{\alpha}(0,r) = \sup_{r\in[0,\infty)} \frac{r^{\alpha}}{\alpha+r^{\alpha}} \Lambda(0,r)\le \sup_{r\in[0,\infty)}\Lambda(0,r) = \sup_{r\in[0,\infty)}\left ( 1+\lambda(0,r)\right )<1-\delta^2$. Since this last bound is independent of $\alpha$, we may take $\alpha\to 0$ (first for $r\neq 0$, then extending to $r=0$ by continuity), yielding
\begin{equation}
\label{eq5.27}
\lambda(t,r)\le -1+ \left ( 1-\delta^2 \right )e^{-2(n-1)\delta^2 t}
\end{equation}
on $[0,T_{\delta}]\times [0,\infty)$, and thus $T_{\delta}=T_{\rm max}$.
\end{proof}

\section{Evolution of $\kappa$}
\setcounter{equation}{0}

\subsection{Evolution equation for $\kappa$}

\begin{lemma}\label{lemma6.1}
The evolution equation for the sectional curvature $\kappa$ is
\begin{equation}
\label{eq6.1} \frac{\partial \kappa}{\partial t} =\Delta \kappa +\left [ n-1+\kappa
+(n-2)\lambda \right ] \left ( r\frac{\partial \kappa}{\partial r} +2\kappa\right )
+\frac{2(n-2)}{r^2f^2}\left ( \lambda -\kappa \right )\ .
\end{equation}
\end{lemma}

\begin{proof}
Starting from (\ref{eq4.3}), we obtain
\begin{equation}
\begin{split}
\label{eq6.2} \frac{\partial}{\partial t} \left ( -\frac{1}{f^2} \right )
=&\, \Delta \left ( -\frac{1}{f^2} \right )+\left ( \frac{\partial}{\partial r}
\left ( -\frac{1}{f^2} \right )\right )^2+\left [ \frac{(n-2)}{r}-\frac{n}{rf^2}+(n-1)r\right ] \frac{\partial}{\partial r}
\left ( -\frac{1}{f^2} \right )\\
&\, +\frac{2(n-2)}{r^2}\left [ \left ( -\frac{1}{f^2} \right )
+\left ( -\frac{1}{f^2} \right )^2\right ]\ .
\end{split}
\end{equation}
Differentiating both sides of (\ref{eq6.2}) with respect to $r$, multiplying by $\frac{1}{2r}$, and using that $\frac{1}{2r}\frac{\partial}{\partial r}\left ( -\frac{1}{f^2} \right )=\frac{1}{rf^3}\frac{\partial f}{\partial r}=\kappa$, then we get that
\begin{equation}
\begin{split}
\label{eq6.3}\frac{\partial \kappa}{\partial t}=&\, \frac{1}{2r}\frac{\partial}{\partial r}
\left ( \Delta\left ( -\frac{1}{f^2} \right ) \right )+4r\kappa \frac{\partial \kappa}{\partial r}
+4\kappa^2+\left [ \frac{(n-2)}{r} - \frac{n}{rf^2} +(n-1)r\right ] \frac{\partial \kappa}{\partial r}\\
&\,
+2\left [ n-1+n\kappa \right ] \kappa +\frac{2(n-2)}{r^4f^2}\left ( 1-\frac{1}{f^2} \right ) +\frac{2(n-2)}{r^2}
\left ( 1-\frac{2}{f^2}\right )\kappa \ .
\end{split}
\end{equation}
On the other hand, we have that
\begin{equation}
\label{eq6.4}
\frac{1}{2r}\frac{\partial}{\partial r}
\left ( \Delta\left ( -\frac{1}{f^2} \right ) \right )
=\frac{1}{2r} \frac{\partial}{\partial r} \left [ \frac{1}{r^{n-1}f}\frac{\partial}{\partial r}
\left ( \frac{r^{n-1}}{f}\frac{\partial}{\partial r} \left ( -\frac{1}{f^2} \right ) \right )\ ,
\right ]
\end{equation}
whose right-hand side can be re-expressed to yield
\begin{equation}
\label{eq6.5}
\frac{1}{2r}\frac{\partial}{\partial r}
\left ( \Delta\left ( -\frac{1}{f^2} \right ) \right )=\Delta \kappa +\frac{2}{rf^2} \frac{\partial \kappa}{\partial r}
-3r\kappa\frac{\partial \kappa}{\partial r}-2(n+1)\kappa^2\ .
\end{equation}
Substituting this into (\ref{eq6.3}) yields
\begin{equation}
\begin{split}
\label{eq6.6} \frac{\partial\kappa}{\partial r}
=&\, \Delta \kappa +r\kappa \frac{\partial \kappa}{\partial r} +\left [ \frac{(n-2)}{r}
\left ( 1-\frac{1}{f^2} \right ) +(n-1)r\right ] \frac{\partial \kappa}{\partial r}\\
&\, +2\left [ \kappa +n-1+\frac{(n-2)}{r^2}\left ( 1-\frac{2}{f^2} \right ) \right ] \kappa
+\frac{2(n-2)}{r^4f^2}\left ( 1-\frac{1}{f^2}\right ) \ .
\end{split}
\end{equation}
Simplifying and using (\ref{eq3.13}) (the definition of $\lambda$), we obtain (\ref{eq6.1}).
\end{proof}

\subsection{The difference of sectional curvatures}

\begin{lemma}\label{lemma6.2} The difference of sectional curvatures obeys the evolution equation
\begin{equation}
\label{eq6.7}
\begin{split}
\frac{\partial}{\partial t}(\kappa-\lambda)=&\, \Delta (\kappa-\lambda) +\left [ n-1+\kappa+(n-2)\lambda\right ]r\frac{\partial}{\partial r}(\kappa-\lambda)\\
&\, +2\left [ 2(n-1)\lambda -\frac{n}{r^2} +n-1\right ] (\kappa-\lambda)\ .
\end{split}
\end{equation}
\end{lemma}

\begin{proof}
Subtract (\ref{eq5.7}) from (\ref{eq6.1}) and write the resulting equation in terms of the difference $\kappa - \lambda$. The Bianchi identity (\ref{eq3.14}) is then used to convert some of the $r\frac{\partial\lambda}{\partial r}$ factors into $2(\kappa-\lambda)$. After some manipulation, (\ref{eq6.7}) follows.
\end{proof}

Now define
\begin{equation}
\label{eq6.8}
\mu:=e^{4a(n-1)t}(\kappa-\lambda)^2
\end{equation}
for $a\in {\mathbb R}$. Then Lemma \ref{lemma6.2} yields
\begin{equation}
\label{eq6.9}
\begin{split}
\frac{\partial\mu}{\partial t}=&\, \Delta \mu -e^{4a(n-1)t}\left \vert \nabla (\kappa-\lambda) \right \vert^2+ \left [ n-1+\kappa +(n-2)\lambda \right ] r\frac{\partial\mu}{\partial r}\\
 &\, +4\left [ 2(n-1)\lambda -\frac{n}{r^2}+(n-1)(1+a) \right ]\mu\\
\le &\, \Delta\mu +\left [ n-1+\kappa +(n-2)\lambda \right ] r\frac{\partial\mu}{\partial r}+4(n-1)(1+a+2\lambda)\mu\ ,
\end{split}
\end{equation}
where the inequality follows since by definition we have $\mu\ge 0$. Writing out the Laplacian in coordinates, we obtain

\begin{equation}
\label{eq6.10} \frac{\partial\mu}{\partial t}\leq \frac{1}{f^2}\frac{\partial^2 \mu}{\partial r^2} +\left ( (n-1)r + \frac{(n-1)}{r}-r\lambda\right ) \frac{\partial\mu}{\partial r}+4(n-1)(1+a+2\lambda)\mu\ .
\end{equation}

The $1/r$ multiplying the first derivative produces a singularity, but only of the mild sort expected from our choice of polar coordinate system. Nonetheless, we will have to handle it, and since we do not have \emph{a priori} uniform control over $\mu$ at $r=0$, we follow the technique used in the last section and define
\begin{equation}
\label{eq6.11} \mu_\alpha:=\frac{r^{\alpha}}{(r^{\alpha}+\alpha)}\mu\ ,\ \alpha\in (0,1)\ ,
\end{equation}
so that $\mu_\alpha(t,0)= 0$.
Using the straightforward expressions
\begin{equation}
\label{eq6.12} \frac{\partial \mu}{\partial r} = \left ( 1+\frac{\alpha}{r^{\alpha}} \right ) \frac{\partial \mu_\alpha}{\partial r}-\frac{\alpha^2}{r^{\alpha+1}}\mu_\alpha
\end{equation}
and
\begin{equation}
\label{eq6.13}\frac{\partial^2 \mu}{\partial r^2} = \left ( 1+\frac{\alpha}{r^{\alpha}} \right ) \frac{\partial^2 \mu_\alpha}{\partial r^2}-\frac{2\alpha^2}{r^{\alpha+1}}\frac{\partial \mu_\alpha}{\partial r} +\frac{\alpha^2(\alpha+1)}{r^{\alpha+2}}\mu_\alpha\ ,
\end{equation}
and simplifying a term using the definition of $\lambda$ (\ref{eq3.13}), we obtain an inequality for the evolution of $\mu_\alpha$:

\begin{equation}
\label{eq6.14}
\begin{split}
\frac{\partial\mu_{\alpha}}{\partial t}\le &\, \frac{1}{f^2}\frac{\partial^2 \mu_{\alpha}}{\partial r^2} +\left [ (n-1)r + \frac{(n-1)}{r}-r\lambda-\frac{2\alpha^2}{r(r^{\alpha}+\alpha)f^2}\right ] \frac{\partial \mu_{\alpha}}{\partial r}\\
&\, +\left [ 4(n-1)(1+a+2\lambda)-\frac{\alpha^2}{(r^{\alpha}+\alpha)}\left ( \frac{(n-2-\alpha)}{r^2}+\alpha\lambda + n-1\right ) \right ]\mu_{\alpha}\ .
\end{split}
\end{equation}
Using $n\ge 3$, taking $\alpha\in (0,1)$, and using that $\mu_{\alpha}\ge 0$ by its definition, this simplifies to
\begin{equation}
\label{eq6.15}
\frac{\partial\mu_{\alpha}}{\partial t}\le \frac{1}{f^2}\frac{\partial^2 \mu_{\alpha}}{\partial r^2} +\left [ (n-1)r + \frac{(n-1)}{r}-r\lambda-\frac{2\alpha^2}{r(r^{\alpha}+\alpha)f^2}\right ] \frac{\partial \mu_{\alpha}}{\partial r}
+{\hat c}(t,r) \mu_{\alpha} \ ,
\end{equation}
where
\begin{equation}
\label{eq6.16}
{\hat c}(t,r) := 4(n-1)(1+a+2\lambda)-\frac{\alpha^3\lambda }{(r^{\alpha}+\alpha)}\ .
\end{equation}

\subsection{Evolution and convergence of $\mu$}

\noindent We can now prove that $\mu(t)\to 0$ as $t\to\infty$, provided that $\lambda(0)<0$. As before, when $\lambda(0)\le -1$ we will find rapid convergence, whereas in the more general case the convergence may be slower. Even more generally, if we allow that $\lambda(0)\le 0$, we will see that $\mu$ remains bounded at any finite time, but we will be unable to show convergence unless $\lambda(0)<0$.

\begin{proposition}\label{proposition6.3}
Given $\lambda(t,r)\le -1$ for all $(t,r)\in [0,T_{\max})\times [0,\infty)$, then
\begin{equation}
\label{eq6.17} \left \vert \kappa(t,r)-\lambda(t,r)\right \vert \le e^{-2b(n-1)t} \max_{r\in[0,\infty)} \left \vert \kappa(0,r)-\lambda(0,r)\right \vert
\end{equation}
for any $b<1$.
\end{proposition}

\begin{proof} Choose any $0<T<T_{\max}$. As we have that $\mu_{\alpha}(t,0)=0$ and $\mu_{
\alpha}(t,r)\to 0$ as $r\to \infty$, Proposition \ref{proposition4.1} implies that either $\mu_{\alpha}$ is identically zero (and then $\mu$ is identically zero, and the Bianchi identity implies the manifold is hyperbolic space), or the maximum of $\mu_{\alpha}(t,r)$ on $[0,T]\times [0,\infty)$ is positive and occurs at $t=0$, or the maximum of $\mu_{\alpha}(t,r)$ is positive and occurs at a point $(t^*,r^*)\in (0,T]\times (0,\infty)$ where ${\hat c}(t^*,r^*)\ge 0$. Since $\lambda(t,r)\le -1$ we find $0 \leq \frac{-\alpha^3 \lambda}{r^{\alpha} + \alpha} \leq -\alpha^2 \lambda$, and then (\ref{eq6.16}) yields
\begin{equation}
\label{eq6.18}
\begin{split}
{\hat c}(t,r) \le &\, 4(n-1)\left[1+a+\left ( 2-\frac{\alpha^2}{4(n-1)}\right ) \lambda \right]\\
\le &\, 4(n-1)\left [a+\left (1-\frac{\alpha^2}{4(n-1)}\right )\lambda \right ]\ .
\end{split}
\end{equation}
Pick $\alpha$ sufficiently small, say $0<\alpha<1$. Since $\lambda\le -1$, the right-hand side of (\ref{eq6.18}) is strictly negative if $a<1-\frac{\alpha^2}{4(n-1)}$. As we can choose $\alpha>0$ as close to zero as we wish, any $a<1$ suffices to make ${\hat c}<0$. Therefore, for sufficiently small $\alpha$, the minimum of $\mu_{\alpha}$ occurs at $t=0$. Now take $\alpha\searrow 0$ and then take $T\nearrow T_{\max}$. Finally, to express the decay in the form of (\ref{eq6.17}), take a square root.
\end{proof}

We now weaken the condition $\lambda(0)\le -1$. We begin with the following time-dependent bound which implies a long-time existence result.

\begin{proposition}\label{proposition6.4}
Given $\lambda(0,\cdot)\le 0$ for all $(t,r)\in [0,\infty)\times [0,\infty)$, then $\kappa(t,\cdot)+1$ is bounded on $[0,T]$ for any $T>0$.
\end{proposition}

\begin{proof} First, from Propositions \ref{proposition5.1} and \ref{proposition5.2}, $\lambda(t,r)$ is bounded above by $0$ and bounded below by a function that decays to $-1$ exponentially. Thus we can find some $K>0$ which depends only on the initial sectional curvature $\lambda(0,\cdot)$ such that $0 \geq \lambda(t,r)\ge -4(n-1)K$. Then estimating (\ref{eq6.16}) with $\alpha\in (0,1)$ we have
\begin{equation}
\label{eq6.19}
\begin{split}
{\hat c}(t,r) \le &\, 4(n-1)\left[1+a+\left ( 2-\frac{\alpha^2}{4(n-1)}\right ) \lambda \right]\\
\le &\, 4(n-1)(1+a)+4(n-1)K\ ,
\end{split}
\end{equation}
Choose $a=-K-2$. Then ${\hat c}(t,r)\le -4(n-1)<0$, and thus each $\mu_{\alpha}$ governed by (\ref{eq6.15}) is bounded above by an $\alpha$-independent bound which depends only on $\mu(0,\cdot)$. As before, we can therefore take $\alpha\searrow 0$, thereby obtaining a bound for $\mu(t,r)$. Call the bound $C^2$. Then from (\ref{eq6.8}) we have
\begin{equation}
\label{eq6.20}
\left \vert \kappa -\lambda \right \vert :=e^{-2a(n-1)t}\sqrt{\mu}\le Ce^{2(n-1)(K+2)t}\ .
\end{equation}
Since $\lambda+1$ is bounded on $[0,T]$ for any $T>0$ and since $\left \vert \kappa -\lambda \right \vert$ is also bounded on $[0,T]$ for any $T>0$, then so is $\kappa+1$. \end{proof}

Now we are able to show that if $\lambda(0,\cdot)<0$, then $\kappa(t)$ decays to $-1$. The idea is that Proposition \ref{proposition6.4} would have yielded the requisite decay of $\kappa$ if we could have chosen a positive value for $a$ in (\ref{eq6.19}). Since (\ref{eq6.19}) is obtained from (\ref{eq6.16}) by estimating $\lambda$, we return to (\ref{eq6.16}) and observe that when $\lambda$ is sufficiently close to $-1$ and $\alpha$ is sufficiently small, there are positive choices for $a$ which will make $c<0$. If $\lambda$ is not initially close enough to $-1$, we can invoke Propositions \ref{proposition6.4} and \ref{proposition5.3}, which tell us that eventually $\lambda(t,\cdot)$ will be close enough to $-1$.

\begin{proposition}\label{proposition6.5}
Assume that $\lambda(0,r)< 0$ for all $r\in [0,\infty)$. Then there are constants $\gamma>0$ and $C>0$ depending only on $g(0,\cdot)$ and $n$ such that $\vert \kappa(t,r)+1\vert < C e^{-\gamma t}$ for all $(t,r) \in [0,\infty)\times [0,\infty)$.
\end{proposition}

\begin{proof}
Take $\lambda(0,r)\le -M<0$ for some $M>0$ (since $\lambda(0,r)<0$ and $\lambda(0,r)\to -1$ as $r\to\infty$). As in the preamble, by Proposition \ref{proposition6.4} the flow will exist for all $t>0$. By Proposition \ref{proposition5.3} there will be a time $T(M,\epsilon)\ge 0$ after which $|\lambda(t)+1|<\epsilon$, where we take $\epsilon$ to be small enough for our purposes, say $\epsilon=1/10$. During the interval $[0,T(M,\epsilon)]$, the magnitude of $\kappa+1$ may have increased, but is bounded by the estimate in Proposition \ref{proposition6.4}.

Then equation (\ref{eq6.16}) yields
\begin{equation}
\label{eq6.21}
\begin{split}
{\hat c}(t,r)\leq &\, 4(n-1)\left[1+a+\left ( 2-\frac{\alpha^2}{4(n-1)}\right ) \lambda \right]\\
< &\, 4(n-1) \left ( a-\frac{4}{5} + \frac{9 \alpha^2}{40(n-1)}\right )\ ,
\end{split}
\end{equation}
and once more choosing $\alpha$ near $0$ allows us to make a positive choice for $a$ yielding ${\hat c}(t,r)<0$. The rest of the argument is as before, yielding that on the domain $[T(M,1/10),\infty)$ $\mu$ decays as
\begin{equation}
\label{eq6.22}
\mu\le C e^{-2a(n-1) t}\ ,
\end{equation}
for some constant $C$ which depends on $\mu(T(M,1/10))$ and $n$, and thus only on $M$ and $n$. Finally, because both $\mu$ and $\lambda+1$ decay exponentially, so does $\kappa+1$.
\end{proof}

\section{Proofs of the main theorems}
\setcounter{equation}{0}

\noindent Since Theorem \ref{theorem1.1} was proved in Section 2, it remains to prove Theorems \ref{theorem1.2} and \ref{theorem1.4}. In what follows, it will help to recall that
\begin{equation}
\label{eq7.1}
\left \vert {\rm Rm} +K\right \vert^2 = 4(n-1) \left [ (\kappa+1)^2+\frac12 (n-2) (\lambda+1)^2 \right ]
\end{equation}
in rotational symmetry. For the purposes of Theorem \ref{theorem1.1}, we interpret the left-hand side as being evaluated at $F_t^{-1}(r,\dots)$; i.e., we evaluate the left-hand side along the normalized Ricci flow obtained by pulling back our diffeomorphism-modified flow (\ref{eq3.9}) by the diffeomorphisms $F_t$ of the preceding section.

\subsection{Proof of Theorem \ref{theorem1.2}}

\begin{proof}
Statement (a) follows from Theorem \ref{theorem1.1}. By Propositions \ref{corollary4.4} and \ref{proposition5.1}, the combination $\lambda+1$ as defined in terms of equation (\ref{eq3.13}) remains bounded on any finite interval of time. Then by Proposition \ref{proposition6.4}, so does the combination $\kappa+1$ with $\kappa$ given by (\ref{eq3.12}). But by equation (\ref{eq7.1}) and the unique evolution posited by Theorem \ref{theorem1.1} then the solution $g(t)$ of the corresponding normalized Ricci flow exists for all time, so statement (b.i) is proved. Moreover, the conclusion of Corollary \ref{corollary4.4} holds under pullback to the normalized Ricci flow, so (b.ii) is also proved. The decay estimate (b.iii) similarly follows from Propositions \ref{proposition5.3} and \ref{proposition6.5}.
\end{proof}

\subsection{Proof of Theorem \ref{theorem1.4}}

\begin{proof}
Pick any increasing, divergent sequence of times $\{ t_k\big \vert k=1,2,\dots ,\ t_1>0 \}$ along the flow, and set $g_k = g(t_k)$.  By Propositions \ref{proposition5.3} and \ref{proposition6.5} the sectional curvatures are uniformly bounded along any such sequence. By Corollary \ref{corollary2.4}, all covariant derivatives of the curvatures are also uniformly bounded.

Since $({\mathbb R}^n,{g}_{k})$ is rotationally symmetric, choose the origin of symmetry and call it $O_k$. There can be no geodesic loops through $O_k$ since geodesics through the origin of a rotationally symmetric manifold are necessarily radial, and radial geodesics for a rotationally symmetric metric on ${\mathbb R}^n$ do not close. Then $\inj_{O_k}=\conj_{O_k}$ for each $k$. But $\conj_{O_k}=\infty$, for if a point $p$ at distance $R>0$ from $O_k$ were conjugate to $O_k$, then by rotational symmetry the sphere of radius $R$ about $O_k$ would consist entirely of points conjugate to $O_k$, and no geodesic through $O_k$ could minimize beyond this sphere, contradicting the Hopf-Rinow theorem on ${\mathbb R}^n$.

Thus by \cite[Theorem 2.3]{Hamilton}, we can find a convergent subsequence. Furthermore, by Propositions \ref{proposition5.3} and \ref{proposition6.5} the limit metric must have constant sectional curvature $-1$.

\end{proof}

\end{document}